%% file: NonClassicalForms_v10.tex
\documentclass[a4paper]{amsart}

\input{DefaultAMS}

\usepackage{longtable}

\calCapital %
\frakSmall %
\bbCapital %

\newcommand{\ids}[1]{\mathrm{E}(#1)}

\newcommand{\lAd}[1]{\mathrm{Ad}^\ell_{#1}}
\newcommand{\rAd}[1]{\mathrm{Ad}^r_{#1}}

\newcommand{\mul}[1]{\odot_{#1}} %{\circledast_{#1}}
\newcommand{\DRMod}[1]{{\mathrm{DMod}\textrm{-}{#1}}}

\newcommand{\aEnd}{\End^-}
\newcommand{\aAut}{\Aut^-}

\newcommand{\regBil}[1]{\Bil_{\mathrm{reg}}({#1})}
\newcommand{\genBil}[1]{\Bil_{\mathrm{gen}}({#1})}

\newcommand{\lotimes}[1]{{}_{#1}\!\otimes}

\newcommand{\twistmul}[1]{\diamond_{#1}}

\newcommand{\wsim}{\sim_{\mathrm{w}}}
\newcommand{\wiso}{\cong_{\mathrm{w}}}

\title{General Bilinear Forms}
\author{Uriya A.\ First}
\date{\today}
\address{Einstein Institute of Mathematics, Hebrew University of Jerusalem}
\email{uriya.first@gmail.com}
\thanks{This research was partially supported by an Israel-US BSF grant \#2010/149 and an ERC grant \#226135.}

\begin{document}

\maketitle

\begin{abstract}
    We introduce the new notion of \emph{general bilinear forms} (generalizing sesquilinear forms)
    and prove that for
    every ring $R$ (not necessarily commutative, possibly without involution)
    and every right $R$-module $M$ which is a \emph{generator}
    (i.e.\ $R_R$ is a summand of $M^n$ for some $n\in\N$),
    there is a one-to-one correspondence between the anti-automorphisms of $\End(M)$
    and the general \emph{regular} bilinear forms on $M$, considered up to \emph{similarity}.
    This generalizes a well-known similar correspondence in the case $R$ is a field. We also demonstrate that there
    is no such correspondence for arbitrary $R$-modules.

    We use the generalized correspondence to show that there is a
    canonical set isomorphism $\Inn(R)\setminus\! \aAut(R)\cong \Inn(\nMat{R}{n})\setminus\! \aAut(\nMat{R}{n})$,
    provided $R_R$ is the only right $R$-module $N$ satisfying $N^n\cong R^n$, and also to prove
    a variant of a theorem of Osborn. Namely, we classify all semisimple rings
    with involution admitting no non-trivial idempotents that are invariant under the involution.

\end{abstract}

\section{Overview}

    Unless specified otherwise, all rings are assumed to have a unity and ring
    homomorphisms are required to preserve it.
    Subrings are assumed to have the same unity as the ring containing them.
    Given a ring $R$, denote its set of invertible elements by $\units{R}$, its center by $\Cent(R)$ and
    its inner automorphism group by $\Inn(R)$. The $n\times n$
    matrices over $R$ are denoted by $\nMat{R}{n}$ and category of right (left) $R$-modules is denoted by $\rMod{R}$
    ($\lMod{R}$). If a module $M$ can be considered as a module over several rings, we use $M_R$ to denote ``$M$,
    considered as a right $R$-module''.
    Throughout, a \emph{semisimple} ring means a \emph{semisimple artinian} ring.

    \medskip

    Let $R$ be an arbitrary (not-necessarily commutative) ring and let
    $M$ be a right $R$-module satisfying certain mild assumptions (e.g.\ being a generator).
    In the main result of this paper, we  establish a one-to-one correspondence between the
    set of anti-automorphisms
    of the ring $\End_R(M)$, denoted
    $\aAut(\End_R(M))$,  and the regular bilinear forms on $M$, considered up to an suitable equivalence relation.
    Moreover, the correspondence  maps involutions to symmetric bilinear forms.
    The statements just made assumed that there exists a notion of bilinear  forms on modules over \emph{arbitrary} rings.
    Indeed, to that purpose we introduce \emph{general bilinear forms}, which generalize sesquilinear forms and other similar notions.

    We will use the generalized correspondence to show that there is a canonical set isomorphism
    \begin{equation}\label{GEN:EQ:anti-auto-of-matrices-promo}
    \Inn(R)\setminus\! \aAut(R)\quad\cong\quad \Inn(\nMat{R}{n})\setminus\! \aAut(\nMat{R}{n})\ ,
    \end{equation}
    provided that $N^n\cong R^n$ implies $N\cong R_R$ for all $N\in\rMod{R}$ (with $n\in\N$ fixed). In case
    $R$ has an anti-automorphism (e.g.\ if $R$ is commutative),
    this implies that $\Inn(R)\setminus\!\Aut(R)\cong \Inn(\nMat{R}{n})\setminus\!\Aut(\nMat{R}{n})$, a
    statement that can be understood as a Skolem-Noether Theorem; compare with \cite[Th.\ 2.10]{Sa99}.

    We also use the correspondence to give an easy proof to a variant of a theorem of Osborn
    (\cite[Th.\ 2]{Osborn70}). Osborn's Theorem determines the structure of rings
    with involution $(R,\alpha)$ in which $2$ is invertible and all $\alpha$-invariant elements are invertible or nilpotent.
    We will determine the structure of all \emph{semisimple} rings with involution $(R,\alpha)$ such that the only $\alpha$-invariant
    idempotents in $R$ are $0$ and $1$. In particular, we will get a new proof of Osborn's Theorem in the case $R$ is semilocal.

    Further more specialized applications (e.g.\ \cite{Fi13B}) will be published elsewhere.

    \medskip

    Our correspondence in fact generalizes a similar well-known correspondence in the case $R$ is a field:

\rem{
    Let $(R,*)$ be a ring with involution and let $\lmb\in \Cent(R)$ be an element satisfying
    $\lmb^*\lmb=1$. Recall that a \emph{sesquilinear space} over $(R,*)$ consists of a pair
    $(M,b)$ such that $M$ is a right $R$-module and $b:M\times M\to R$ is a biadditive
    map satisfying
    \[
    b(xr,y)=r^*b(x,y)\qquad\textrm{and}\qquad b(x,yr)=b(x,y)r
    \]
    for all $x,y\in M$ and $r\in R$.
    If moreover $b(x,y)=\lmb b(y,x)^*$, then $b$ is called \emph{$\lmb$-hermitian} (or \emph{$(*,\lmb)$-hermitian}).
    Hermitian forms and sesquilinear forms are widely studied in the literature, sometimes
    under the context of \emph{hermitian categories}; for instance, see \cite{OMeara63}, \cite{QuSchSch79},
    \cite{SchQuadraticAndHermitianForms}, \cite{Kn91}. To avoid ambiguity with general bilinear forms,
    we henceforth address sesquilinear forms as \emph{classical bilinear forms} (over $(R,*)$).

    The following theorem is well-known:
}

    \begin{thm}\label{GEN:TH:classical-cor}
        Let $F$ be a field, let $\alpha_0$ be an automorphism
        of $F$ of order $1$ or $2$, and let $V$ be a finite dimensional vector space.
        Then there exists a one-to-one correspondence between nondegenerate sesquilinear
        forms $b:V\times V\to F$ over $(F,\alpha_0)$, considered up to multiplication by an element of $\units{F}$,
        and the anti-automorphisms $\alpha$ of $\End_F(V)$ whose restriction to $F=\Cent(\End_F(V))$
        is $\alpha_0$. The correspondence is given by sending $b:V\times V\to F$ to the unique anti-automorphism $\alpha$
        satisfying
        \[
        b(w x,y)=b(x,w^\alpha y)\qquad\forall\ x,y\in M,\, w\in\End_F(V)
        \]
        and it maps involutions to hermitian forms.
    \end{thm}

    \begin{proof}
        See \cite[Ch.\ 1, Th.\ 4.2]{InvBook}.
    \end{proof}

    We will present another proof of Theorem~\ref{GEN:TH:classical-cor}, which do not use the Skolem-Nother Theorem
    (as is the case in all proofs we have seen). Note that our correspondence treats all anti-automorphisms
    of $\End_F(V)$ and not only those whose order on $F$ is $1$ or $2$.

    \medskip

    Section \ref{section:FORM:definitions} defines general bilinear forms and presents some examples.
    In sections \ref{section:GEN:preface}
    and \ref{section:GEN:definitions}, we construct and discuss the correspondence described above,
    where in section \ref{section:GEN:regularity}, we prove that our construction does yield a correspondence under mild assumptions.
    Section~\ref{section:GEN:orthogonal-sums} examines how the new correspondence interacts with orthogonal sums, and this
    is used in section~\ref{section:GEN:generators} to sharpen the results of  section~\ref{section:GEN:regularity}.
    In that section, we show the isomorphism in \eqref{GEN:EQ:anti-auto-of-matrices-promo} and
    give a  proof of Theorem~\ref{GEN:TH:classical-cor}. Section~\ref{section:GEN:osborn}
    proves the variant of Osborn's Theorem presented above.
    Finally,
    section~\ref{section:GEN:examples}
    brings various examples of modules over which our correspondence fails.

\section{General Bilinear Forms}
\label{section:FORM:definitions}

    In this section, we define general bilinear forms and study their basic properties.
    Throughout, $R$ is a (possibly non-commutative) ring.

    \begin{dfn} \label{FORM:DF:double-module}
        A \emph{(right) double $R$-module} is an additive group $K$ together with two
        operations $\mul{0},\mul{1}:K\times R\to K$ such that $K$ is a right $R$-module
        with respect to each of $\mul{0}$, $\mul{1}$ and
        \[(k\mul{0}a)\mul{1}b=(k\mul{1}b)\mul{0}a\qquad \forall ~k\in K,~a,b\in R\ .\]
        We let $K_i$ denote the $R$-module obtained by letting $R$ act on $K$ via $\mul{i}$.

        The class of (right) double $R$-modules is denoted by $\DRMod{R}$.
        For $K,K'\in\DRMod{R}$, we define $\Hom(K,K')=\Hom_R(K_0,K'_0)\cap\Hom_R(K_1,K'_1)$.
        This makes $\DRMod{R}$ into an abelian category. (The category
        $\DRMod{R}$ is isomorphic to\linebreak
        $\rMod{(R\otimes_\Z R)}$ and also to the category
        of $(R^\op,R)$-bimodules; see Remark~\ref{GEN:RM:why-double} below for why do we prefer double modules.)\rem{\footnote{
            We prefer double $R$-modules
            on $(R^\op,R)$-bimodules since
            the usage of $(R^\op,R)$-bimodules often requires twisting of the left $R^\op$-module
            into a right $R$-module structure, thus causing ambiguity as to which right $R$-module
            structure is used.
        }}
    \end{dfn}

    \begin{dfn}
        Let $K$ be a double $R$-module. An \emph{anti-automorphism} of $K$ is
        an additive bijective  map $\theta:K\to K$ satisfying
        \[(k \mul{i} a)^\theta=k^\theta \mul{1-i} a \qquad\forall a\in R,~k\in K,~i\in\{0,1\}\ .\]
        If additionally $\theta\circ \theta=\id_K$, then $\theta$ is called an \emph{involution}.
    \end{dfn}

    A (general) \emph{bilinear space} over $R$ is a triplet $(M,b,K)$ such that $M\in\rMod{R}$, $K\in\DRMod{R}$
    and $b:M\times M\to K$ is a biadditive map
    satisfying
    \[b(xr,y)=b(x,y)\mul{0} r\qquad\textrm{and}\qquad b(x,yr)=b(x,y)\mul{1}r\]
    for all $x,y\in M$ and $r\in R$.
    In this case, $b$ is called a (general) \emph{bilinear form} (over $R$).
    Let $\theta$ be an involution of $K$.
    The form $b$ is called \emph{$\theta$-symmetric} if
    \[b(x,y)=b(y,x)^\theta\qquad\forall x,y\in M\ .\]

    \begin{example}\label{GEN:EX:standard-classical-bilinear-forms}
        Let $(R,*)$ be a ring with involution.
        Recall that a
        \emph{sesquilinear space} over $(R,*)$ consists of a pair
        $(M,b)$ such that $M$ is a right $R$-module and $b:M\times M\to R$ is a biadditive
        map satisfying
        \[
        b(xr,y)=r^*b(x,y)\qquad\textrm{and}\qquad b(x,yr)=b(x,y)r
        \]
        for all $x,y\in M$ and $r\in R$.
        If moreover $b(x,y)=\lmb b(y,x)^*$, then $b$ is called \emph{$\lmb$-hermitian} (or \emph{$(*,\lmb)$-hermitian}).
        (See \cite{QuSchSch79} or \cite{Kn91} for an extensive discussion about sesquilinear and hermitian forms.)

        We can make any  sesquilinear  form
        $b:M\times M\to R$  fit into our definition of general bilinear forms;
        simply turn
        $R$ into a double $R$-module by defining $r\mul{0}a=a^*r$ and $r\mul{1}a = ra$ for all $a,r\in R$.
        Moreover, for every $\lmb\in \Cent(R)$ with $\lmb^*\lmb=1$,  the map $\theta:R\to R$
        defined by $r^{\theta}:=\lambda r^*$ is an involution of $R$ (as a double $R$-module),
        and $b$ is $\lmb$-hermitian if and only if $b$ is $\theta$-symmetric.
    \end{example}

    To avoid any ambiguity, we shall henceforth address sesquilinear forms as \emph{classical bilinear forms}.
    General bilinear forms obtained as in Example~\ref{GEN:EX:standard-classical-bilinear-forms} will be called
    classical as well.

    \begin{example}\label{GEN:EX:form-over-incidence-algebra}
        Let $F$ be a field and
        let $S:\nMat{F}{n\times m}\to\nMat{F}{m\times n}$ be given by
        \[A^S=\SMatIII{}{}{1}{}{\iddots}{}{1}{}{}\trans{A}\SMatIII{}{}{1}{}{\iddots}{}{1}{}{}\ .\]
        (Here, $T$ denotes the transpose operation;
        $S$ reflects $A$ along the diagonal emanating from its top-right corner). It is easy
        to verify that $(AB)^S=B^SA^S$ for any two matrices over $F$, provided the multiplication is defined.

        Let $R$, $K$ be the sets of matrices of the forms
        \[
        \left[
        \begin{smallmatrix}
        * & * & * \\
        * & * & * \\
        0 & 0 & *
        \end{smallmatrix}
        \right]\quad,\quad
        \left[
        \begin{smallmatrix}
        * & * & * \\
        0 & 0 & * \\
        0 & 0 & *
        \end{smallmatrix}
        \right]
        \]
        with entries in $F$, respectively. Then $R$ is a subring of $\nMat{F}{3}$ and $K$
        is a double $R$-module w.r.t.\ the operations
        \[
        k\mul{0} r=r^Sk\qquad\textrm{and}\qquad k\mul{1}r=kr\ ,
        \]
        where $k\in K$ and $r\in R$. Furthermore, the map $\theta:k\mapsto k^S$
        is an involution of $K$. Let $M$ be the set of matrices of the form
        $
        [
        \begin{smallmatrix}
        * & * & * \\
        0 & 0 & *
        \end{smallmatrix}
        ]
        $
        with entries in $F$. Then $M$ is a right $R$-module w.r.t.\ the standard matrix multiplication.
        Define $b:M\times M\to K$ by $b(x,y)=x^Sy$. Then $(M,b,K)$ is a bilinear space over $R$ and $b$
        is $\theta$-symmetric. Nevertheless, $R$ has no anti-automorphism! This can be seen by carefully checking
        how would an anti-automorphism act on the the standard matrix units in $R$, or by noting that $R$
        is an \emph{incidence algebra} of a partially-ordered set without an anti-automorphism.
    \end{example}

    We now turn to define nondegenerate and regular\footnote{
        Some texts use ``unimodular'' instead of ``regular''.
    }
    bilinear forms.
    Henceforth, $K$ is a fixed double $R$-module.

    \smallskip

    Let $M\in\rMod{R}$ and let $i\in\{0,1\}$. The \emph{$i$-$K$-dual} (or just $i$-dual) of
    $M$ is defined by\footnote{
        The reason that we do not define $M^{[i]}$ to be $\Hom_{R}(M,K_i)$ is because
        we want $R^{[i]}$ to be isomorphic to $K_i$ via $f\leftrightarrow f(1)$.
    }
    \[M^{[i]}:=\Hom_{R}(M,K_{1-i})\ .\]
    Note $M^{[i]}$ is naturally
    a right $R$-module w.r.t.\ the operation $(fr)(m)=(fm)\mul{i}r$ (where $f\in M^{[i]}$, $r\in R$ and $m\in M$).
    Moreover,
    $M\mapsto M^{[i]}$ is a left-exact contravariant functor from $\rMod{R}$ to itself, which we denote by $[i]$.

    Let $b:M\times M\to K$ be a (general) bilinear form. The \emph{left adjoint} and \emph{right adjoint}
    of $b$ are defined as follows:
    \[\lAd{b}:M\to M^{[0]},\quad (\lAd{b} x)(y)=b(x,y)\ ,\]
    \[\rAd{b}: M\to M^{[1]},\quad (\rAd{b} x)(y)=b(y,x)\ ,\]
    where $x,y\in M$. It is straightforward to check that $\lAd{b}$
    and $\rAd{b}$ are right $R$-linear. We say that
    $b$ is \emph{right regular} if $\rAd{b}$ is bijective and
    \emph{right injective} or \emph{right nondegenerate} if $\rAd{b}$ is injective.
    Left regularity and left injectivity are defined in the same manner. Bilinear forms
    that are not right injective are called \emph{right degenerate}. This is equivalent
    to the existence of $0\neq x\in M$ such that $b(M,x)=0$.

    If $b$ is right regular, then every $w\in\End_R(M)$ admits a \emph{unique} element $w^\alpha\in\End_R(M)$
    such that
    \[
    b(wx,y)=b(x,w^\alpha y)\qquad\forall x,y\in M\ .
    \]
    Indeed, a straightforward computation shows that $w^\alpha=(\rAd{b})^{-1}\circ w^{[1]}\circ \rAd{b}$
    satisfies these requirements. (Recall that $w^{[1]}:M^{[1]}\to M^{[1]}$ is defined by $w^{[1]}f=f\circ w$.)
    The map $w\mapsto w^\alpha$, denoted $\alpha$, is easily
    seen to be anti-\emph{endo}morphism\footnote{
        An anti-endomorphism of a ring is an additive map that preserves the unity and reverses the order of multiplication.
        It is not required to be bijective (in which case it is called an anti-automorphism).
    } of $\End_R(M)$ and is thus called the (right) \emph{corresponding anti-endomorphism} of $b$.
    Our usage of the term anti-\emph{endo}morphism, rather than anti-\emph{auto}morphism, is essential here
    because $\alpha$ need not be bijective; see Example \ref{FORM:EX:base-example}. Nevertheless, in
    case $b$ is also \emph{left} regular, $\alpha$ is invertible, for
    the left regularity implies that for all $w\in W$ there is $w^{\beta}\in W$ such that $b(x,wy)=b(w^\beta x,y)$,
    and the map $\beta$ (called the \emph{left} corresponding anti-endomorphism of $b$)
    is easily seen to be the inverse of $\alpha$.
    Furthermore, if $b$ is $\theta$-symmetric
    for some involution $\theta:K\to K$, then $\alpha$ is an involution. Indeed,
    \[
    b(x,wy)=b(wy,x)^\theta=b(y,w^\alpha x)^\theta=b(w^\alpha x,y)=b(x,w^{\alpha\alpha}y)
    \]
    and this forces $w=w^{\alpha\alpha}$ (since $b$ is right regular).

    \medskip

    We now define asymmetry maps. We will only need them briefly in sections~\ref{section:GEN:definitions} and~\ref{section:GEN:generators},
    but in general, asymmetries
    are important tools in studying non-symmetric
    forms (see \cite{Ri74}, \cite{RiSh76}, \cite{ScharR81} for classical applications).
    Let $\theta$ be an anti-auto\-mor\-phism of $K$.
    A \emph{right $\theta$-asymmetry} (resp.\ left $\theta$-asymmetry) of $b$ is a map $\lmb\in \End(M_R)$ such that $b(x,y)^\theta=b(y,\lmb x)$
    (resp.\ $b(x,y)^\theta=b(\lmb y,x)$)
    for all $x,y\in M$.
    A right $\theta$-asymmetry need neither exist nor be unique.
    Nevertheless, the following holds:

    \begin{prp}\label{FORM:PR:basic-props-of-forms-I}
        Let $(M,b,K)$ be a bilinear space over $R$ and let $\theta$ be an anti-auto\-mor\-phism of $K$.
        Define $u_{\theta,M}:M^{[0]}\to M^{[1]}$ by $u_{\theta,M}(f)=\theta\circ f$.
        Then:
        \begin{enumerate}
            \item[(i)] $u_{\theta,M}$ is an isomorphism of right $R$-modules.\footnote{
                In fact, $\{u_{\theta,M}\}_{M\in\rMod{R}}$ is a natural isomorphism from the functor $[0]$ to the functor $[1]$.
            }
            \item[(ii)] $\lmb\in\End_R(M)$ is a right $\theta$-asymmetry $\iff$
            $u_{\theta,M}\circ \lAd{b}=\rAd{b}\circ \lmb$.
            \item[(iii)] If $b$ is right regular, then $(\rAd{b})^{-1}\circ u_{\theta,M}\circ \lAd{b}$ is a right $\theta$-asymmetry of $b$.
            \item[(iv)] If $b$ is right injective, then $b$ admits at most one right $\theta$-asymmetry.
            \item[(v)] The inverse of an invertible right $\theta$-asymmetry of $b$ is a left $\theta^{-1}$-asymmetry of $b$.
            \item[(vi)] $b$ is $\theta$-symmetric $\iff$ $\id_M$ is a right $\theta$-asymmetry of $b$.
            \item[(vii)] Assume $b$ is $\theta$-symmetric. Then $b$ is right regular (injective) if and only if
            $b$ is left regular (injective).
        \end{enumerate}
    \end{prp}

    \begin{proof}
        (i)--(vi) follow by straightforward computation. (vii) follows from (i), (ii) and (vi).
    \end{proof}

    \begin{example} \label{FORM:EX:base-example}
        Let $R$ be a ring and let $\alpha$ be an anti-\emph{endo}morphism of $R$. Define
        $K$ to be the double $R$-module obtained from $R$ by setting
        \[
        k\mul{0} r=r^\alpha k \qquad
        \textrm{and}
        \qquad k\mul{1}r=kr
        \]
        for all $r,k\in R$. The double $R$-module $K$ is called the \emph{standard double $R$-module
        of $(R,\alpha)$}.
        Define $b:R\times R\to K$ by $b(x,y)=x^\alpha y$. Then $b$ is a bilinear
        form. As $b(R,x)=0$ implies $x=0$ (since $x=b(1,x)=0$), $b$ is right injective.
        In addition, it is straightforward to check for all $f\in R^{[1]}=\Hom_R(R_R,K_0)$, $\rAd{b}(f(1))=f$, hence $\rAd{b}$ is also surjective.
        Therefore, $b$ is right regular.

        Now observe that
        for all $r,x,y\in R$,
        $b(rx,y)=(rx)^\alpha y=x^\alpha r^\alpha y=b(x,r^\alpha y)$. Thus, once identifying $\End(R_R)$
        with $R$ via $f\leftrightarrow f(1)$, the corresponding anti-endomorphism of $b$ is $\alpha$.
        It is also straightforward to check that
        $\ker(\lAd{b})=\ker(\alpha)$ and $\im(\lAd{b})=\im(\alpha)$ (once identifying $R^{[0]}=\Hom_R(R_R,K_1)=\End(R_R)$
        with $R$ as before). Thus, $\lAd{b}$ is left injective (surjective) if and only if
        $\alpha$ is. In particular, if $\alpha$ is neither injective nor surjective, then so is $\lAd{b}$, despite
        the fact
        $\rAd{b}$ is bijective. In this case, the corresponding anti-endomorphisms of $b$ is not bijective.
    \end{example}

    \begin{example}
        Let $S,R,M,b,K$ be as in Example \ref{GEN:EX:form-over-incidence-algebra}.
        Then a straightforward (but tedious) computation shows that $b$ is right and left regular.
        Moreover, $\End_R(M)$ can be identified with the ring of $2\times 2$ upper-triangular
        matrices over $F$, acting on $M$ from the left by matrix multiplication. As $b$ is $\theta$-symmetric,
        it induces an involution on $\End_R(M)$, which is easily seen to be the map $S$. Indeed,
        for all $x,y\in M$ and $w\in\End_R(M)$, we have
        \[
        b(wx,y)=(wx)^Sy=x^Sw^Sy=b(x,w^Sy)\ .
        \]
    \end{example}

    \begin{remark}
        There is no obvious way to explain general bilinear forms
        as sesquilinear forms over a hermitian category (see \cite[Ch.\ 7, \S2]{SchQuadraticAndHermitianForms} for
        definitions). However,
        it is possible to generalize the notion of hermitian categories to naturally include general bilinear forms.
        This construction  will be published elsewhere.
    \end{remark}

    \begin{remark}\label{GEN:RM:why-double}
        Our double $R$-modules are nothing but $R\otimes_\Z R$-modules or $(R^{\op},R)$-bimodules.
        However, we are led to use this notation for several reasons.
        First, this notation is shorter and clarifies the computations. (For example,
        consider the elegant definition of the functors $[0]$ and $[1]$: For
        every $M\in\rMod{R}$ and $i\in\{0,1\}$,  define $M^{[i]}=\Hom_R(M,K_{1-i})$
        and make $M^{[i]}$ into
        a right $R$-module by setting $(fr)m=(fm)\mul{i}r$).
        If we had used the language of $(R^{\op},R)$-bimodules, we would often have to  twist the left $R^{\op}$-action of $K$
        into a right $R$-action, thus causing ambiguity as to which right $R$-module structure is used. (For comparison,
        here is the definition of $[0]$ and $[1]$ in the language of bimodules: Given $M\in\rMod{R}$ and an
        $(R^\op,R)$-bimodule $K$, define $M^{[0]}=\Hom_R(M,K_R)$ and $M^{[1]}=\Hom_{R^\op}(M,{}_{R^\op}K)$,
        where in the definition of $[1]$, we twist $M$ into a left $R^\op$-module. $M^{[0]}$ admits a standard left $R^\op$-action, which
        we consider as a right $R$-action. $M^{[1]}$ has a standard right $R$-action. The shifting
        from left $R^\op$-modules to right $R$-modules becomes inconvenient when dealing with
        objects like $M^{[1][0]}$, which are treated below.)
        In addition, it is well known that every left-exact contravariant functor from $\rMod{R}$ to $\lMod{S}$ is of the
        form $\Hom(-,K)$ where $K$ is an $(S,R)$-bimodule.
        Likewise, it is possible to show that any left-exact contravariant functor from $\rMod{R}$ to $\rMod{S}$ is of the form
        $\Hom(-,K)$ where $K$ is an abelian group admitting a right $R$-module
        structure and a right $S$-module structure that commute with each other (in a vague sense, applying $\Hom(-,K)$ ``kills'' the $R$-action
        and leaves the $S$-action).
        Finally, our notation allows a natural generalization to multilinear forms: Define a right multi-$R$-module
        to be an abelian group $(K,+)$ admitting a family of $n$ right $R$-module structures $\mul{1},\dots,\mul{n}:K\times R\to K$ that
        commute with each other. A multilinear form over $R$ would be a map $b:M\times M\times\dots\times M\to K$
        such that $b$ is additive in all components and $b(\dots,x_{i-1},x_ir,x_{i+1},\dots)=b(\dots,x_{i-1},x_i,x_{i+1},\dots)\mul{i}r$
        for all $1\leq i\leq n$, $x_1,\dots,x_n\in M$ and $r\in R$.
    \end{remark}

\section{From Anti-Endomorphisms to Bilinear Forms}
\label{section:GEN:preface}

    Let $R$ be a ring and let $M$ be a  right $R$-module. Set $W=\End_R(M)$ and let\label{GEN:DF:a-end-and-a-aut}
    $\aEnd(W)$ ($\aAut(W)$) denote the set of anti-endomorphisms (anti-automorphisms) of $W$.
    We have seen that any right regular bilinear form $b:M\times M\to K$ induces an anti-endomorphism
    $\alpha\in\aEnd(W)$, which we henceforth denote by $\alpha(b)$. In this section, we construct an ``inverse''
    of the map $b\mapsto \alpha(b)$. That is, for every $\alpha\in\aEnd(W)$, we will define a bilinear space $(M,b_\alpha,K_\alpha)$
    such that
    \[
    b_\alpha(wx,y)=b(x,w^\alpha y)\qquad\forall\, x,y\in M,\, w\in W\ .
    \]
    This remarkable since, to our best knowledge, over fields, there is no canonical way to construct the \emph{classical} bilinear form
    that corresponds to a given anti-automorphism
    (see Theorem~\ref{GEN:TH:classical-cor}). Moreover, the existence of this form is usually shown using ``heavy tools''
    such as the Skolem-Noether Theorem. What allows the
    unexpected shortcut in the \emph{general} case is the freedom in choosing the double $R$-module $K_\alpha$; we
    do not have to identify it with a prescribed double $R$-module.

    \medskip

    Henceforth, $M$ is a fixed right $R$-module and $W=\End(M_R)$.

    \medskip

    We begin by introducing some new notation.
    Let $\alpha\in\aEnd(W)$ and let $A,B$ be two \emph{left} $W$-modules. Define:\label{GEN:DF:alpha-tensor}
    \[A\otimes_\alpha B=\frac{A\otimes_\Z B}{\Trings{wa\otimes b-a\otimes w^\alpha b\where a\in A,b\in B,w\in W}}~.\]
    For $a\in A$ and $b\in B$, we let $a\otimes_\alpha b$ denote the image of $a\otimes_\Z b$ in $A\otimes_\alpha B$ (the subscript $\alpha$
    will be dropped when obvious from the context).

    \begin{remark}\label{GEN:RM:twisting-remark}
        For any $B\in\lMod{W}$ and $\alpha\in\aEnd(W)$, let $B^\alpha$
        denote the right $W$-module
        obtained by twisting $B$ via $\alpha$. Namely, $B^\alpha=B$ as sets,
        but $B^\alpha$ is equipped with a right action $\twistmul{\alpha}:B\times W\to B$
        given by $x\twistmul{\alpha} w=w^\alpha x$ for all $x\in B$ and $w\in W$.
        Then the abelian group $A\otimes_\alpha B$ can
        be naturally identified with $B^\alpha \otimes_W A$.
        Therefore, $\otimes_\alpha$ is a  biadditive bifunctor
        and $W^n\otimes_\alpha B\cong B^n$.
    \end{remark}

    Consider $M$ as a left $W$-module and let $\alpha\in\aEnd(W)$.\label{GEN:DF:b-alpha-K-alpha}
    Define $K_\alpha=M\otimes_\alpha M$ and note that $K_\alpha$ is a double
    $R$-module w.r.t.\ the operations
    \[
    (x\otimes_\alpha y)\mul{0} r=xr\otimes_\alpha y\qquad\mathrm{and}\qquad (x\otimes_\alpha y)\mul{1} r=x\otimes_\alpha yr\
    \]
    ($x,y\in M$, $r\in R$). It is now clear that the map $b_\alpha:M\times M\to K_\alpha$ defined
    by $b_\alpha(x,y)=x\otimes_\alpha y$ is a bilinear form and
    \begin{equation} \label{GEN:EQ:stableness-for-b-alpha}
    b_\alpha(w x,y)=w x\otimes_\alpha y=x\otimes_\alpha w^\alpha y=b_\alpha(x,w^\alpha y)
    \end{equation}
    for all $x,y\in M$ and $w\in W$, hence $\alpha(b_\alpha)=\alpha$, provided $b_\alpha$ is right regular.
    In fact, the pair
    $(b_\alpha,K_\alpha)$ is universal w.r.t.\ satisfying \eqref{GEN:EQ:stableness-for-b-alpha} in the sense that if $b:M\times M\to K$
    is a bilinear form satisfying $b(wx,y)=b(x,w^\alpha y)$ for all $w\in W$, then there is a unique double
    $R$-module homomorphism
    $f:K_\alpha\to K$ such that $b=f\circ b_\alpha$. It is given by
    $
        f(x\otimes_\alpha y)=b(x,y)
    $.

    Assume further that $\alpha$ is an involution.
    Then $K_\alpha$ admits an involution $\theta_\alpha$ given by $\theta_\alpha(x\otimes_\alpha y)= y\otimes_\alpha x$, and
    $b_\alpha$ is $\theta_\alpha$-symmetric. Thus, every involution induces a symmetric form!

    \begin{example}\label{GEN:EX:classical-involutions}
        Let $F$ be a field, let $V$ be a f.d.\ $F$-vector-space and let $\alpha$ be an anti-automorphism
        of $\End_F(V)$ of order $1$ or $2$ on $F=\Cent(\nMat{F}{n})$. We will show below (Proposition~\ref{GEN:PR:classical-cor-over-fields}) that
        $b_\alpha$ is regular and
        $K_\alpha$ is isomorphic to the standard double $F$-module of $(F,\alpha|_F)$
        (see Example \ref{FORM:EX:base-example}). In particular, when identifying
        $K_\alpha$ with $F$, $b_\alpha$ becomes a \emph{classical} regular bilinear form
        over $(F,\alpha|_F)$, and that form (necessarily) corresponds to $\alpha$ in
        Theorem \ref{GEN:TH:classical-cor}.
        Moreover, if $\alpha$ is an \emph{orthogonal} involution, then $\theta_\alpha=\id_F$ and
        $b_\alpha$ is symmetric, and if $\alpha$ is a \emph{symplectic} involution, then $\theta_\alpha=-\id_F$
        and $b_\alpha$ is alternating.
    \end{example}

    Recalling Theorem \ref{GEN:TH:classical-cor}, we now ask whether
    the maps $b\mapsto \alpha(b)$ and $\alpha\mapsto b_\alpha$ give rise to a one-to-one correspondence
    between the right regular bilinear forms on $M$, considered up to a \emph{suitable equivalence relation}, and the
    anti-endomorphisms of $W$. In contrast to  Theorem \ref{GEN:TH:classical-cor}, the answer
    is no in general (regardless of the equivalence relation chosen), because $b_\alpha$ need not be right regular.

    \begin{example}\label{GEN:EX:counter-example-I}
        Consider the $\Z$-module $M=\Z[\frac{1}{p}]/\Z$. It is well-known that $\End(M_\Z)=\Z_p$,
        where $\Z_p$ are the $p$-adic integers. (This follows from Matlis' Duality Theory; see \cite{Ma58} or
        \cite[\S3I]{La99}.)
        Take $\alpha=\id_{\Z_p}\in\aEnd(\Z_p)$, and note that the module $M$ is $p$-divisible. Therefore, for all $x,y\in M$,
        \[x\otimes_\alpha y=p^np^{-n}x\otimes_\alpha y
        =p^{-n}x\otimes_\alpha \alpha(p^n)y=p^{-n}x\otimes_\alpha p^{n}y\ .\]
        (The ``quotient'' $p^{-n}x$ is not uniquely determined, but it does not matter to us.)
        As $p^ny=0$ for sufficiently large $n$, it follows that $x\otimes y=0$. This implies $K_\alpha=0$, hence $b_\alpha=0$.
        Moreover, the universal property of $b_\alpha$ means that there is no bilinear form $0\neq b':M\times M\to K'$
        satisfying $b'(wx,y)=b'(x,w^\alpha y)$ for
        all $w\in \Z_p$ and $x,y\in M$. In particular, $\alpha$ does not correspond to a right regular form on $M$.
    \end{example}

    More examples of this flavor can be found at section \ref{section:GEN:examples}.
    We leave the problem of determining when is $b_\alpha$ right regular to section~\ref{section:GEN:regularity},
    and proceed with defining the equivalence relation on the \emph{class} of right regular bilinear forms on $M$.\footnote{
        This is a class rather than a set because the bilinear forms in it can take
        values in arbitrary double $R$-modules.
    } There is an obvious candidate for this relation:

    \begin{dfn}
        Call two bilinear forms $b:M\times M\to K$ and $b':M\times M\to K'$
        \emph{similar} if
        there is an isomorphism $f\in \Hom_{\DRMod{R}}(K,K')$ such that $b'=f\circ b$. In this
        case, $f$ is called a \emph{similarity} from $b$ to $b'$ and we write $b\sim b'$.
    \end{dfn}

    It is easy to see that two similar regular bilinear forms induce the same anti-endomorphism.
    In addition,
    for classical bilinear forms over fields, being similar coincides with
    being the same up to multiplying by a nonzero scalar, which is the equivalence relation
    used in Theorem~\ref{GEN:TH:classical-cor}.

    \medskip

    Let us conclude: Denote by $\regBil{M}$ the \emph{category} of regular bilinear forms on $M$ with similarities as morphisms.
    We want to have a one-to-one correspondence as follows:
    \begin{equation}\label{GEN:EQ:correspondence-I}
            \xymatrixcolsep{5pc}\xymatrix{
            {\regBil{M}/\!\sim} \ar@/^1pc/[r]^{b\,\mapsto\alpha(b)} &  {\aEnd(W)} \ar@/^1pc/[l]^{b_\alpha \mapsfrom\, \alpha}
            }\ .
    \end{equation}
    In order of that to happen, we need to show that:
    \begin{enumerate}
        \item[(1)] The map $\alpha\mapsto b_\alpha$ takes values in $\regBil{M}$. Namely, $b_\alpha$ is right regular for all $\alpha\in\aEnd(W)$.
        \item[(2)] The maps $\alpha\mapsto b_\alpha$ and $b\mapsto \alpha(b)$ are mutual inverses. As
        $\alpha(b_\alpha)=\alpha$ when $b_\alpha$ is regular, this amounts to showing $b\sim b_{\alpha(b)}$ for all $b\in \regBil{M}$.
    \end{enumerate}
    (Note that (1) does imply (2); see Example~\ref{GEN:EX:b-is-not-similar-to-b-alpha}.)

    The main result of this paper (Theorem~\ref{GEN:TH:main-thm-II}) asserts that both (1) and (2) hold
    when $M$ is a \emph{progenerator}, and the analogous weaker claims for (right and left) regular
    forms and anti-\emph{auto}morphisms of $W$ hold when $M$ is a \emph{generator}. We will also show that (1) holds under other mild assumptions,
    e.g.\ when $M$ is finitely generated projective.

    \begin{remark}\label{GEN:RM:dfn-of-generic-forms}
        When (1) holds and (2) fails,
        it is still possible to obtain a correspondence between forms and anti-endomorphisms by specializing
        to \emph{generic} forms. A bilinear form $b:M\times M\to K$ is
        called \emph{right generic} if it is right regular and similar to $b_\alpha$ for some $\alpha\in\aEnd(W)$. In this case,
        we must have $\alpha(b)=\alpha$, implying $b_{\alpha(b)}=b_\alpha\sim b$. Letting $\genBil{M}$ denote
        the category of right generic bilinear forms on $M$ with similarities as morphisms, we easily see that there is a one-to-one
        correspondence
        \begin{equation}\label{GEN:EQ:correspondence-adjusted-II}
            \xymatrixcolsep{5pc}\xymatrix{
            {\genBil{M}/\!\sim} \ar@/^1pc/[r]^{b\,\mapsto\alpha(b)} &  {\aEnd(W)} \ar@/^1pc/[l]^{b_\alpha \mapsfrom\, \alpha}
            }\ ,
        \end{equation}
        provided $b_\alpha$ is right regular for all $\alpha\in\aEnd(W)$. The restriction to $\genBil{M}$ is not so bad because
        every right regular bilinear form $b$ can be swapped with its \emph{generization}, defined to be $b_{\alpha(b)}$.
        (Notice that we are assuming $b_\alpha$ is always right regular and hence the generization of $b$ is right regular. However,
        it is still open whether the generization of a right regular form is right regular in general.)
    \end{remark}

    Right generic forms are well-behaved behaved in comparison to right regular forms. This is demonstrated in the following proposition,
    which fails completely for regular forms (Example~\ref{GEN:EX:b-is-not-similar-to-b-alpha}).

    \begin{prp}\label{GEN:PR:generic-forms-basic-props}
        Let $b:M\times M\to K$ be a right generic bilinear form. If $\alpha(b)$ is an involution, then $K$ has an involution
        $\theta$ and $b$ is $\theta$-symmetric. Furthermore, $b$ is left regular.
    \end{prp}

    \begin{proof}
        We can identify $K$ with $K_{\alpha(b)}$ and $b$ with $b_{\alpha(b)}$. Then $b_{\alpha(b)}$ is $\theta_{\alpha(b)}$-symmetric
        as explained above. Since $b$ is right regular by assumption and $\theta_{\alpha(b)}$-symmetric, it is also left regular
        by Proposition~\ref{FORM:PR:basic-props-of-forms-I}(vii).
    \end{proof}

    %The next sections will be dedicated to determining when is $b_\alpha$ right regular and to provide
    %explicit description of the double $R$-module $K_\alpha$ in certain cases.

    \begin{remark}\label{GEN:RM:onto-remark}
        For a bilinear space $(M,b,K)$, let $\im(b)$ denote the additive group spanned
        by $\{b(x,y)\where x,y\in M\}$.
        (Caution: In general, $\im(b)$ is not the image of
        $b$ in the usual sense.)
        It is easy to see that $\im(b)$ is a sub-double-$R$-module
        of $K$ and that $\im(b_\alpha)=K_\alpha$ for all $\alpha\in\aEnd(W)$.
        However, $\im(b)$ might be strictly smaller than $K$, even when $b$ is regular, in
        which case $b$ is necessarily not similar to its generization.
        This observation suggests that the problem of  regular bilinear
        forms which are not similar to their generizations might be solved
        by restricting to bilinear spaces $(M,b,K)$ with $\im(b)=K$. However, this
        adjustment is not enough in general; the regular bilinear spaces $(M,b,K)$ constructed in
        Example \ref{GEN:EX:b-is-not-similar-to-b-alpha} satisfy $\im(b)=K$ and $b\nsim b_{\alpha(b)}$.
    \end{remark}

    \begin{remark}\label{GEN:RM:weak-similarity}
        Call two right stable bilinear forms \emph{weakly similar} (denoted $\wsim$) if they have
        similar generizations. Then under the assumption that $b_\alpha$ is right regular for all $\alpha\in\aEnd(W)$,
        there is a one-to-one correspondence between $\regBil{M}/\!\wsim$ and $\aEnd(W)$.
        However, the author could not find a natural way to make $\regBil{M}$ into a category whose
        isomorphism classes are the equivalence classes of $\wsim$, i.e.\ defining \emph{weak similarities}.
        \rem{(As a thumb rule, a definition of weak similarities would be appropriate if it applies
        to arbitrary bilinear forms, rather than just right regular forms.)}
    \end{remark}

    We finish this section by presenting a \emph{left} analogue of $b_\alpha$.
    Assume $A,B\in\lMod{W}$. In the same manner we have defined $A\otimes_\alpha B$, we define
    \[A\lotimes{\alpha} B=\frac{A\otimes_\Z B}{\Trings{a\otimes wb-w^\alpha a\otimes  b\where a\in A,b\in B,w\in W}}~.\]
    In addition, we define ${}_\alpha K=M\lotimes{\alpha} M$
    and ${}_\alpha b:M\times M\to {}_\alpha K$ by $b(x,y)=x\lotimes{\alpha} y$. All the results of this
    paper have \emph{left versions} obtained by replacing $b_\alpha$, $K_\alpha$ with ${}_\alpha b$, ${}_\alpha K$
    and every right property with its left version.

    We also note that if $\alpha$ is bijective, then $A\otimes_\alpha B$ is naturally isomorphic to
    $A\lotimes{\alpha^{-1}} B$ (via $x\otimes_\alpha y\leftrightarrow x\lotimes{\alpha^{-1}} y$)
    and $b_\alpha$ is similar to ${}_{\alpha^{-1}}b$, hence both right and left versions of our results apply.
    %We also note that $A\lotimes{\alpha} B\cong A^\alpha\otimes_W B$. (This is similar to Remark \ref{GEN:RM:twisting-remark}.)

\section{Basic Properties}
\label{section:GEN:definitions}

    Let $R$, $M$ and $W$ be as in the previous section and let $\alpha,\beta\in \aEnd(W)$.
    In this  section, we answer the following questions: Provided
    $b_\alpha$ and $b_\beta$ are right regular,
    \begin{enumerate}
        \item[(1)] when are $K_\alpha$ and  $K_\beta$ isomorphic?
        \item[(2)] when are $b_\alpha$ and $b_\beta$ \emph{weakly isometric} (see below)?
        \item[(3)] when does $K_\alpha$ have an anti-automorphism? an involution?
    \end{enumerate}
    The answers are phrased in terms of $W$ and turn out to be independent of $R$ and $M$.
    They agree with the approach of \cite{InvBook} and other texts that, roughly,
    isomorphism classes of anti-automorphisms (resp.\ involutions) are in correspondence with
    isometry classes of
    sesquilinear (resp.\ hermitian) forms considered up to scalar multiplication.

    Throughout, $\Inn(W)$ denotes the group of \emph{inner} automorphisms
    of $W$ (i.e.\ those given by conjugation with an invertible element of $W$).

    \begin{prp}\label{GEN:PR:K-alpha-isomorphic-to-K-beta}
        Let $\alpha\in\aEnd(W)$ and $\vphi\in\Inn(W)$. Then $K_{\alpha}\cong K_{\vphi\circ \alpha}$ as
        double $R$-modules. Conversely, if $\alpha,\beta\in\aEnd(W)$ are such that $b_{\alpha}$ and $b_\beta$ are right regular
        and $K_\alpha\cong K_\beta$, then there exists $\vphi\in \Inn(W)$ such that
        $\beta=\vphi\circ \alpha$. In particular, if $b_\alpha$ is right regular for all $\alpha\in \aEnd(W)$,
        then the isomorphism classes of the modules $K_\alpha$ correspond to the orbits of
        the left action of $\Inn(W)$ on $\aEnd(W)$, i.e.\ to the set $\Inn(W)\setminus\aEnd(W)$.
    \end{prp}

    \begin{proof}
        Throughout, $x,y\in M$ and $w\in W$.
        Let $u\in \units{W}$ be such that $\vphi (w)=uwu^{-1}$ for all $w\in W$.
        Define $f:K_\alpha\to K_{\alpha\circ\vphi}$ by $f(x\otimes_\alpha y)=x\otimes_{\vphi\circ\alpha} uy$.
        Then $f$ is well defined since
        \[
        f(wx\otimes_\alpha y)=
        wx\otimes_{\vphi\circ \alpha} uy=x\otimes_{\vphi\circ \alpha} (uw^\alpha u^{-1})uy=
        x\otimes_{\vphi\circ \alpha}uw^\alpha y=
        f(x\otimes w^\alpha y)
        \]
        and it is easy to see that $f$ is an isomorphism of double $R$-modules
        (its inverse is given by $x\otimes_{\vphi\circ \alpha} y\mapsto x\otimes_\alpha u^{-1}y$).
        Therefore, $K_\alpha\cong K_{\vphi\circ \alpha}$.

        To prove the second part of the proposition, it is enough to show that if\linebreak $b,c:M\times M\to K$
        are two right regular bilinear forms with corresponding anti-endomorphisms $\alpha$ and $\beta$,
        then there exists $\vphi\in \Inn(W)$ s.t.\ $\beta=\vphi\circ \alpha$. Indeed,
        define $u=(\rAd{b})^{-1}\circ \rAd{c}\in \units{W}$. Then for all $x,y\in M$, $c(x,y)=(\rAd{c}y)x=
        (\rAd{b}(uy))x=b(x,uy)$. Therefore, for all $w\in W$:
        \[
        c(x,w^\beta y)=
        c(wx,y)=b(wx,uy)=
        b(x,w^\alpha uy)=
        c(x,u^{-1}w^\alpha uy)
        \]
        and it follows that $w^\beta=u^{-1}w^\alpha u$, as required.
    \end{proof}

    Call two general bilinear spaces $(M,b,K)$, $(M',b',K')$ \emph{weakly isometric}, denoted $b\wiso b'$, if there
    exist isomorphisms $\sigma\in \End_R(M,M')$ and $f\in \Hom_{\DRMod{R}}(K,K')$
    such that $b'(\sigma x,\sigma y)=f(b(x,y))$. In this case $(\sigma,f)$ is called a \emph{weak isometry}
    from $b$ to $b'$. For example, two classical bilinear forms over a field $F$ are weakly isometric
    if and only if they are isometric after multiplying one of them by a nonzero scalar.

    Let $(S,\gamma),(S',\gamma')$ be two rings with  anti-endomorphism. Recall that a homomorphism
    of rings with anti-endomorphism from $(S,\gamma)$ to $(S',\gamma')$ is a ring homomorphism
    $\vphi:S\to S'$ such that $\vphi\circ \gamma=\gamma'\circ \vphi$. In case
    $S=S'$, $\vphi$ is called \emph{inner} if it is inner as a ring endomorphism of $S$.

    \begin{prp}\label{GEN:PR:b-alpha-weakly-isometric-to-b-beta}
        Let $\alpha,\beta\in\aEnd(W)$. If there exists an inner isomorphism $\vphi:(W,\alpha)\to (W,\beta)$,
        then $b_\alpha\wiso b_\beta$. Conversely, if $b_\alpha$ and $b_\beta$ are right regular and
        weakly isomorphic, then there exists an inner isomorphism $\vphi:(W,\alpha)\to (W,\beta)$.
    \end{prp}

    \begin{proof}
        Let $\vphi:(W,\alpha)\to (W,\beta)$ be an inner isomorphism and write
        $\vphi(w)=uwu^{-1}$ for a suitable $u\in W$.
        The equality $\vphi\circ \alpha=\beta\circ \vphi$ is clearly
        equivalent to $w^\alpha =\vphi^{-1}(\vphi(w)^\beta)=(u^\beta u)^{-1} w^\beta (u^\beta u)$
        for all $w\in W$. Therefore, by the proof of Proposition \ref{GEN:PR:K-alpha-isomorphic-to-K-beta},
        the map $f:K_\alpha\to K_\beta$ given by
        $f(x\otimes_\alpha y)=x\otimes_\beta u^\beta uy$ is a well-defined isomorphism
        of double $R$-modules. It is now routine to check that $(u,f)$ is a weak
        isometry from $b_\alpha$ to $b_\beta$.

        Now assume we are given a weak isometry $(u,f):b_\alpha\to b_\beta$.
        Then for all $x,y\in  M$ and $w\in W$, we have
        \begin{eqnarray*}
        f(b_\alpha(wx,y))&=& b_\beta(uwx,uy)
        =b_\beta(x,w^\beta u^\beta uy)=
        b_\beta(u^{-1}ux,w^\beta u^\beta uy)\\
        &=&
        b_\beta(ux,(u^{-1})^\beta w^\beta u^\beta uy)=
        f(b_\alpha(x,u^{-1}(u^\beta)^{-1}w^\beta u^\beta uy))\\
        &=&
        f(b_\alpha(x,(u^\beta u)^{-1}w^\beta (u^\beta u)))\ .
        \end{eqnarray*}
        Since $f$ is bijective and $b_\alpha$ is right regular, $w^\alpha=(u^\beta u)^{-1} w^\beta (u^\beta u)$, and
        this is equivalent to $\vphi\circ \alpha=\beta\circ \vphi$ with $\vphi(w):=uwu^{-1}$.
    \end{proof}

    \begin{prp}\label{GEN:PR:basic-properties-of-b-alpha-II}
        Let $\alpha\in\aEnd(W)$ and assume there
        exists $\lambda\in W$ such that $w^{\alpha\alpha}\lmb=\lmb w$ for all $w\in W$
        and $\lmb^\alpha\lmb\in\units{W}$ (e.g.\ if $\alpha^2\in\Inn(W)$). Then
        the map $\theta:K_\alpha\to K_\alpha$ given by $(x\otimes_\alpha y)^\theta=y\otimes_\alpha \lmb x$
        is a well defined anti-automorphism of $K_\alpha$ and $\lmb$ is a right $\theta$-asymmetry
        of $b_\alpha$. Moreover, if $\lmb^\alpha\lmb=1$, then $\theta$ is an involution.
        Conversely, if $b_\alpha$ is right regular and $K_\alpha$ has an anti-automorphism (or involution) $\theta$,
        then there exists $\lmb\in W$ as above and $\theta$ is induced from $\lmb$.
    \end{prp}

    \begin{proof}
        Throughout, $w\in W$, $r\in R$ and $x,y\in M$.
        The map $\theta$ is well defined since
        \[
        (wx\otimes_\alpha  y)^\theta=
        y\otimes_\alpha \lmb wx =
        y\otimes_\alpha w^{\alpha\alpha}\lmb x=
        w^\alpha y\otimes_\alpha \lmb x=
        (x\otimes_\alpha w^\alpha y)^\theta\ .
        \]
        To see that $\theta$ is invertible, it is enough to check that $\theta^2$ is invertible. This
        holds since
        \begin{equation}\label{GEN:EQ:asymmetry-eq}
        (x\otimes_\alpha y)^{\theta\theta}=\lmb x\otimes_\alpha \lmb y=x\otimes_\alpha \lmb^\alpha\lmb y
        \end{equation}
        and the map $x\otimes_\alpha y\mapsto x\otimes_\alpha \lmb^\alpha\lmb y$
        has an inverse given by $x\otimes_\alpha y\mapsto x\otimes_\alpha (\lmb^\alpha\lmb)^{-1} y$.
        (The latter is well defined since $\lmb^\alpha\lmb$, and hence $(\lmb^\alpha\lmb)^{-1}$, commutes with $\im(\alpha)$.
        Indeed, $\lmb^\alpha\lmb w^\alpha=\lmb^\alpha(w^\alpha)^{\alpha\alpha}\lmb=(w^{\alpha\alpha}\lmb)^\alpha\lmb=
        (\lmb w)^\alpha\lmb=w^\alpha\lmb^\alpha\lmb$.) That $(k\mul{i} r)^\theta=k^\theta\mul{1-i} r$
        for all $k\in K$
        is straightforward and hence $\theta$ is an anti-automorphism. In addition, \eqref{GEN:EQ:asymmetry-eq}
        also implies that $\theta$ is an involution if $\lmb^{\alpha}\lmb=1$.
        That $\lmb$ is a right $\theta$-asymmetry
        of $b_\alpha$ is routine.

        If $b_\alpha$ is right regular and $K_\alpha$ has an anti-isomorphism $\theta$, then
        by Proposition~\ref{FORM:PR:basic-props-of-forms-I}(iii),
        $b_\alpha$ has a right $\theta$-asymmetry $\lmb\in W$.
        Now, for all $w\in W$, we have
        \[
        b_\alpha(x,\lmb w y)=b_\alpha(wy, x)^{\theta}=b_\alpha(y,w^\alpha x)^{\theta}=b_\alpha(w^\alpha x,\lmb y)=
        b_\alpha(x,w^{\alpha\alpha}\lmb y)
        \]
        and since $b_\alpha$ is right regular, this implies $\lmb w=w^{\alpha\alpha} \lmb$.
        In addition, $b_\alpha(x,y)^{\theta\theta}=b_\alpha(y,\lmb x)^{\theta}=b_\alpha(\lmb x,\lmb y)=
        b_\alpha(x,\lmb^\alpha\lmb y)$. This means $\rAd{b}\circ \lmb^\alpha\lmb=u_{\theta}\circ u_{\theta^{-1}}^{-1}\circ\rAd{b}$, where
        $u_\theta=u_{\theta,M}$ is defined as in Proposition \ref{FORM:PR:basic-props-of-forms-I}. As $\rAd{b}$, $u_\theta$, $u_{\theta^{-1}}$
        are invertible, so is $\lmb^\alpha\lmb$. Furthermore, if $\theta$ is and involution, then $\theta=\theta^{-1}$ and we get
        $\lmb^\alpha\lmb=1$.
        Finally, the anti-automorphism $\theta$ is necessarily
        induced from $\lmb$ because
        \[(x\otimes_\alpha y)^\theta=b_\alpha(x,y)^\theta=b_\alpha(y,\lmb x)=y\otimes_\alpha\lmb x\ .\qedhere\]
    \end{proof}

    \begin{remark}\label{GEN:RM:lmb-is-not-invertible}
        The assumptions of Proposition \ref{GEN:PR:basic-properties-of-b-alpha-II} do not imply
        that $\lmb$ is invertible in $W$. An example demonstrating this will be published elsewhere.
        Nevertheless, $\lmb$ is invertible when $\alpha$ is bijective. Indeed, since $\lmb^\alpha\lmb\in \units{R}$,
        $\lmb$ is right invertible. The bijectivity of $\alpha$ implies the existence of an element
        $u\in W$ with $u^\alpha =\lmb$. We then have $\lmb u=u^{\alpha\alpha}\lmb=\lmb^\alpha\lmb\in\units{R}$,
        so $\lmb$ is also left invertible.
    \end{remark}

    \begin{cor}\label{GEN:CR:right-regular-iff-left-regular}
        If $\alpha\in\aEnd(W)$ and $\alpha^2$ is inner, then
        $b_\alpha$ is right regular (injective) if and only if $b_\alpha$ is left regular (injective).
    \end{cor}

    \begin{proof}
        By assumption, there is an \emph{invertible} $\lmb\in W$ such
        that $w^{\alpha\alpha}=\lmb w\lmb^{-1}$ for all $w\in W$. In particular,
        this means $w^{\alpha\alpha}\lmb =\lmb w$ and $\lmb^\alpha\lmb\in \units{W}$,
        hence by
        Proposition~\ref{GEN:PR:basic-properties-of-b-alpha-II}, $K_\alpha$ has an involution $\theta$
        and $\lmb$ is a $\theta$-asymmetry of $b_\alpha$. By Proposition~\ref{FORM:PR:basic-props-of-forms-I}(ii),
        $u_{\theta,M}\circ \lAd{b}=\rAd{b}\circ \lmb$ and since $u_{\theta,M}$ and $\lmb$ are invertible, it follows
        that $\rAd{b}$ is regular (injective) if and only if $\lAd{b}$ is.
    \end{proof}

\section{Conditions That Imply $b_\alpha$ Is Right Regular}
\label{section:GEN:regularity}

    Let $R$, $M$ and $W$ be as in  section \ref{section:GEN:preface}.
    In this section, we present conditions on $R$, $M$, $W$ and $\alpha$ that ensure
    $b_\alpha$ is right regular, as well as
    other
    supplementary results. In particular, we establish the correspondence in \eqref{GEN:EQ:correspondence-I}
    in case $M$ is a progenerator.

    \medskip

    Assume momentarily that $W$ and $R$ are arbitrary rings and let
    $\biMod{W}{R}$ denote the category of $(W,R)$-bimodules.
    Let $A\in \rMod{W}$, $B\in\rMod{R}$ and $C\in \biMod{W}{R}$. Then
    $\Hom_R(B,C)$ is a right $W$-module w.r.t.\ the action $(fw)m=f(wm)$ (where $f\in\Hom_R(B,C)$,
    $w\in W$ and $m\in M$) and there is a natural group homomorphism
    \[
    \Gamma=\Gamma_{A,B,C}:\,A\otimes_W \Hom_R(B,C)  \to  \Hom_R(B,A\otimes_W C) \\
    \]
    given by $(\Gamma(a\otimes f))b=a\otimes f(b)$ for all $f\in \Hom_R(B,C)$, $a\in A$ and $b\in B$.

    Now assume $M\in\rMod{R}$, $W=\End(M_R)$ and $\alpha\in \aEnd(W)$. Then $M$ can be viewed
    as a $(W,R)$-bimodule.
    Therefore, we have a map
    \[\Gamma=\Gamma_{M^\alpha,M,M}:M^\alpha\otimes_W\Hom_R(M,M)\to \Hom_R(M,M^\alpha\otimes_W M)\]
    (see Remark \ref{GEN:RM:twisting-remark} for the definition of $M^\alpha$).
    In addition, since $M^\alpha$ has an ``unused'' right $R$-module
    structure, $M^\alpha\otimes_W\Hom_R(M,M)$ and $\Hom_R(M,M^\alpha\otimes_W M)$
    can be considered as right $R$-modules and $\Gamma$ becomes an $R$-module homomorphism.
    The following lemma shows that up to certain identifications, $\Gamma$  is actually $\rAd{b_\alpha}$.

    \begin{lem} \label{GEN:LM:main-lemma-II}
        In the previous notation, there is a commutative diagram of right $R$-modules
        \[
        \xymatrix{
        M^\alpha\otimes_W\Hom_R(M_R,M_R)\ar[r]^{\Gamma}\ar[d]^\psi & \Hom_R(M,M^\alpha\otimes_W M)\ar[d]^{\vphi} \\
        M \ar[r]^{\rAd{b_\alpha}} & M^{[1]}
        }
        \]
        where $M^{[1]}=\Hom_R(M,(K_\alpha)_0)$ and $\psi$, $\vphi$ are bijective.
    \end{lem}

    \begin{proof}
        Let $\psi$ be the identity map $M^\alpha\to M$ (recall that $M^\alpha=M$
        as sets) composed on the standard isomorphism
        \[M^\alpha\otimes_W\Hom(M_R,M_R)=M^\alpha\otimes_W W\cong M^\alpha\ .\]
        Then $\psi$ is given by $\psi(m\otimes_W w)=m\twistmul{\alpha}w=w^\alpha m$  and its inverse is $m\mapsto m\otimes 1$.
        The map $\vphi$ is defined by $\vphi(f)=\delta \circ f$ where $\delta$ is the isomorphism
        $M^\alpha\otimes_W M\to K_\alpha$ given by $x\otimes_W y\mapsto y\otimes_\alpha x$.
        The diagram commutes since
        \begin{eqnarray*}
        (\rAd{b_\alpha}(\psi (x\otimes_W w)))y&=&
        (\rAd{b_\alpha}(w^\alpha x))y=
        b_\alpha(y,w^\alpha x)=
        y\otimes_\alpha w^\alpha x=wy\otimes_\alpha x
        \\
        &=&
        \delta(x\otimes_W wy)=
        \delta((\Gamma(x\otimes_W w))y)=
        (\vphi(\Gamma(x\otimes_W w)))y
        \end{eqnarray*}
        for all $x,y\in M$ and $w\in W$.
    \end{proof}

    It is now of interest to find sufficient conditions for $\Gamma$ to be bijective, or at lest injective.
    This is done in the following lemma. \rem{(Note that the examples of section \ref{section:GEN:examples} below show that
    $\rAd{b_\alpha}$ need not be injective nor surjective and hence so does $\Gamma$).}

    \begin{lem} \label{GEN:LM:main-lemma-I}
        Let $A\in \rMod{W}$, $B\in\lMod{R}$ and $C\in\biMod{W}{R}$. Then:
        \begin{enumerate}
            \item[(i)] If one of the following holds:
            \begin{enumerate}
                \item[(a)] $A$ is finitely generated (abbrev.: f.g.) projective,
                \item[(b)] $A$ is projective and $B$ is f.g.,
                \item[(c)] $B$ is f.g.\ projective,
                \item[(d)] $B$ is projective and $A$ is finitely presented (abbrev.: f.p.).
            \end{enumerate}
            Then $\Gamma$ is bijective.
            \item[(ii)] If $A$ is projective, then $\Gamma$ is injective.
            \item[(iii)] If $B$ is projective and $A$ is f.g., then $\Gamma$ is surjective.
            \item[(iv)] If there is an exact sequence $A_1\to A_0\to A\to 0$ and $B$ is projective, then:
            \begin{enumerate}
                \item[(a)] $\Gamma_{A_0,B,C}$ is surjective $\derives$ $\Gamma_{A,B,C}$ is surjective.
                \item[(b)] $\Gamma_{A_0,B,C}$ is bijective and $\Gamma_{A_1,B,C}$ is surjective $\derives$ $\Gamma_{A,B,C}$ is bijective.
            \end{enumerate}
            \item[(v)] If there is an exact sequence $0\to A\to A_0\to A_1$ and $\Hom_R(B,C)$ is flat (in $\lMod{W}$), then:
            \begin{enumerate}
                \item[(a)] $\Gamma_{A_0,B,C}$ is injective $\derives$ $\Gamma_{A,B,C}$ is injective.
                \item[(b)] $\Gamma_{A_0,B,C}$ is bijective, $\Gamma_{A_1,B,C}$ is injective and ${}_WC$ is flat $\derives$ $\Gamma_{A,B,C}$ is bijective.
            \end{enumerate}
            \item[(vi)] If there is an exact sequence $B_1\to B_0\to B\to 0$ and $A$ is flat, then:
            \begin{enumerate}
                \item[(a)] $\Gamma_{A,B_0,C}$ is injective $\derives$ $\Gamma_{A,B,C}$ is injective.
                \item[(b)] $\Gamma_{A,B_0,C}$ is bijective and $\Gamma_{A,B_1,C}$ is injective $\derives$ $\Gamma_{A,B,C}$ is bijective.
            \end{enumerate}
        \end{enumerate}
        In particular, this implies that:
        \begin{enumerate}
            \item[(vii)] If $A$ embeds in a free module and $\Hom_R(B,C)$ is flat, then $\Gamma$ is injective.
            \item[(viii)] If $A$ embeds in a flat module, $B$ is f.g.\ and $\Hom_R(B,C)$ is flat, then $\Gamma$ is injective.
            \item[(ix)] If $A$ is flat and $B$ is f.p., then $\Gamma$ is bijective.
        \end{enumerate}
    \end{lem}

    \begin{proof}
        We prove (i), (ii) and (iii) together: Since $\Gamma=\Gamma_{A,B,C}$ is additive in $A,B,C$ (in
        the functorial sense), we may replace
        projective with free and f.g.\ projective with f.g.\ free.
        Assume $A=\bigoplus_{i\in I}W$, then $\Gamma$ becomes the standard
        map $\bigoplus_{i\in I}\Hom_R(B,C)\to\Hom_R(B,\bigoplus_{i\in I}C)$. This map is clearly injective,
        and provided $I$ is finite, it is bijective. In addition, it is also
        easy to verify it is surjective if $B$ is f.g. Now assume $B=\bigoplus_{i\in I}R$. Then
        $\Gamma$ becomes the standard map $\veps:A\otimes \prod_{i\in I}C\to \prod_{i\in I}(A\otimes C)$,
        which is bijective if $I$ is finite.
        In addition, by \cite[\S4F]{La99}, $\veps$ is surjective if $A$ is f.g.\ and
        bijective if $A$ is finitely presented.

        (iv) We have a commutative diagram with exact rows:
        \[
         \xymatrix{
         A_1\otimes \Hom_R(B,C) \ar[r] \ar[d]^{\Gamma_{A_1,B,C}} & A_0\otimes \Hom_R(B,C) \ar[r] \ar[d]^{\Gamma_{A_0,B,C}} &
         A\otimes \Hom_R(B,C) \ar[d]^{\Gamma_{A,B,C}} \ar[r] & 0 \\
         \Hom_R(B,A_1\otimes C) \ar[r]                 & \Hom_R(B,A_0\otimes C) \ar[r]               & \Hom_R(B,A\otimes C) \ar[r] & 0
         }
        \]
        (The bottom row is exact because $B$ is projective). Then (a) and (b) now follow from the Four Lemma and the Five Lemma, respectively.

        (v) and (vi) are very similar to (iv) and are left to the reader.

        (vii) Let $0\to A\to A_0 \to A_1$ be an exact sequence with $A_0$ free.
        Then $\Gamma_{A_0,B,C}$ is injective by (ii), hence $\Gamma_{A,B,C}$ is injective by (v), since $\Hom_R(B,C)$ is flat.

        (viii) Let $0\to A\to A_0 \to A_1$ be an exact sequence with $A_0$ flat
        and let $B_1\to B_0\to B\to 0$ be a projective resolution with $B_0$ finitely generated.
        Then $\Gamma_{A_0,B_0,C}$ is bijective by (i)-(c), hence  $\Gamma_{A_0,B,C}$
        is injective (by (vi), since $A_0$ is flat), so $\Gamma_{A,B,C}$ is injective (by (v), since $\Hom_R(B,C)$
        is flat).

        (ix) Let $B_1\to B_0\to B\to 0$ be an exact sequence with $B_1$ and $B_0$ being f.g.\ projective.
        Then $\Gamma_{A,B_1,C}$ and $\Gamma_{A,B_0,C}$ are bijective by (i)-(c), hence $\Gamma_{A,B,C}$ is bijective
        (by (vi), since $A$ is flat).
    \end{proof}

    \begin{prp}\label{GEN:PR:main-thm-base}
        Let $M\in\rMod{R}$, $W=\End(M_R)$ and $\alpha\in \aEnd(W)$. Then:
        \begin{enumerate}
            \item[(i)] If $M_R$ or $M^\alpha$ are f.g.\ projective, then $b_\alpha$ is right regular.
            \item[(ii)] If $M^\alpha$ is projective and $M_R$ is f.g., then $b_\alpha$ is right regular.
            \item[(iii)] If $M^\alpha$ embeds in a free right $W$-module, then $b_\alpha$ is right injective.
            \item[(iv)] If $M^\alpha$ embeds in flat right $W$-module and $M_R$ is f.g., then $b_\alpha$ is right injective.
        \end{enumerate}
    \end{prp}

    \begin{proof}
        By Lemma \ref{GEN:LM:main-lemma-II}, that $b_\alpha$ is bijective (injective)
        is equivalent to $\Gamma_{M^\alpha, M,M}$ being bijective (injective). Observe
        that $\Hom_R(M_R,{}_WM_R)\cong {}_WW$  and hence $\Hom_R(M,M)$ is flat. Parts (i)--(iv) of the corollary
        now follow from parts (i)-(a), (i)-(b), (vii), and (viii) of Lemma \ref{GEN:LM:main-lemma-I}, respectively.
    \end{proof}

    \begin{remark}
        Let $\alpha$ be an anti-\emph{auto}morphism of $W$, then $M^\alpha$ is (resp.\ embeds in)
        a free/projective $W$-module if and only if $M$ is. Since any flat module is a direct limit of
        f.g.\ free modules (see \cite{Laz64}) and twisting commutes with direct limits, the previous assertion holds
        upon replacing ``free'' with ``flat''.
    \end{remark}

    As an immediate corollary, we get:

    \begin{cor}\label{GEN:CR:main-thm-I}
        If $M_R$ is f.g.\ projective or f.g.\ semisimple, then $b_\alpha$ is right regular
        for all $\alpha\in\aEnd(W)$. In particular, there exists a one-to-one correspondence
        between $\genBil{M}/\!\sim$ and $\aEnd(W)$ as in \eqref{GEN:EQ:correspondence-adjusted-II}.
    \end{cor}

    \begin{proof}
        The f.g.\ projective case follows from Proposition~\ref{GEN:PR:main-thm-base}(i).
        In case $M_R$ is f.g.\ semisimple, $W$ is semisimple. Therefore, $M^\alpha$ is projective, and the
        corollary follows from Proposition~\ref{GEN:PR:main-thm-base}(ii).
    \end{proof}

    For the next results, recall that the module $M$ is
    a \emph{generator} if every right $R$-module
    is an epimorphic image of $\bigoplus_{i\in I}M$ for some $I$. Equivalently,
    $M$ is a generator if $R_R$ is summand of $M^n$ for some $n\in\N$; see \cite[\S18B]{La99}. The module $M$ is
    a \emph{progenerator} if it is a generator, projective and finitely generated.

    \begin{lem}\label{GEN:LM:when-regular-implies-generic}
        Assume $M$ is a generator and let $(M,b,K)$ be a right regular bilinear space
        with $\alpha=\alpha(b)$. If $b_\alpha$ is right regular, then $b\sim b_\alpha$.
    \end{lem}

    \begin{proof}
        By the universal property of $b_\alpha$, there exists a unique double $R$-module homomorphism
        $f:K_\alpha\to K$ such that $b=f\circ b_\alpha$. We first claim that $f$ is onto. Indeed,
        the image of $f$ is $\im(b)$, where $\im(b)$ is defined as in Remark \ref{GEN:RM:onto-remark}.
        That $M$ is a generator implies that if $\im(b)\subsetneq K$, then $\Hom_R(M,\im(b)_0)\subsetneq \Hom_R(M,K_0)$.
        Since $\rAd{b}$ takes values in the l.h.s., we get a contradiction to the right regularity of $b$, so equality
        must hold. Next, we claim that $f$ is injective. Indeed, since $M$ is a generator, $\ker f=0$
        if and only if $\Hom(M,\ker f)=0$, so it is enough to show the latter.
        Let
        $\vphi\in\Hom_R(M,\ker f)\subseteq\Hom_R(M,(K_\alpha)_0)$. Since $b_\alpha$
        is right regular, there exists $x\in M$ such that $b_\alpha(y,x)=\vphi(y)$ for all $y\in M$. Applying $f$ to both
        sides yields $b(y,x)=f(\vphi(y))=0$, which implies $x=0$ (since $b$ is right regular), hence $\vphi=0$.
    \end{proof}

    \begin{thm}\label{GEN:TH:main-thm-II}
        Assume $M$ is a generator (resp.\ progenerator), then the maps $b\mapsto \alpha(b)$ and
        $\alpha\mapsto b_\alpha$ induce a one-to-one correspondence between the regular (resp.\ right regular) bilinear
        forms on $M$, considered up to similarity, and the elements of $\aAut(W)$ (resp.\ $\aEnd(W)$).
    \end{thm}

    \begin{proof}
        If $M$ is a progenerator, then $b_\alpha$ is right regular for all $\alpha\in\aEnd(W)$ by Corollary~\ref{GEN:CR:main-thm-I}.
        Lemma~\ref{GEN:LM:when-regular-implies-generic} then implies that
        $b_{\alpha(b)}\sim b$ for every regular bilinear form $b$, so we are done.

        If $M$ is a generator, then it
        is well-known that ${}_WM$ is f.g.\ projective (e.g.\ see
        \cite[Exer.\ 4.1.14]{Ro88}). Therefore, for every $\alpha\in\aAut(W)$, $M^\alpha_W$
        is also f.g.\ projective. Now, by Proposition~\ref{GEN:PR:main-thm-base}(i), $b_\alpha$ is right regular, and by
        symmetry (see the end of section \ref{section:GEN:preface}), $b_\alpha$ is also left regular.
        Again, Lemma~\ref{GEN:LM:when-regular-implies-generic} implies that
        $b_{\alpha(b)}\sim b$ for every regular bilinear form $b$. Finally, notice that every $\alpha(b)$ is an invertible
        for every regular bilinear form $b$ (see section~\ref{section:FORM:definitions}) and hence lies in $\aAut(W)$.
    \end{proof}

    We will slightly strengthen Theorem~\ref{GEN:TH:main-thm-II} in section~\ref{section:GEN:generators}.

    \begin{remark}
        Theorem~\ref{GEN:TH:main-thm-II} fails when $M$ is f.g., faithful and projective
        because
        Lemma~\ref{GEN:LM:when-regular-implies-generic} is no longer true
        (Example~\ref{GEN:EX:b-is-not-similar-to-b-alpha}).
        In addition, if $M$ is a generator (and not a progenerator),
        then $b_\alpha$ need not be right regular when $\alpha$ is not bijective (Examples~\ref{GEN:EX:example-IV}
        and~\ref{GEN:EX:example-V}).
    \end{remark}

\rem{
    \begin{thm}
        If $M_R$ is a generator
        and $\alpha$ is an anti-\emph{auto}morphism of $W$, then $b_\alpha$ is regular. In particular,
        there is a one-to-one correspondence between (left and right)  generic forms on $M$, considered up to
        similarity, and anti-\emph{auto}morphisms of $W$.
    \end{thm}

    \begin{proof}
        It is well known that ${}_WM$ is f.g.\ projective when $M_R$ is a generator (e.g.\ see
        \cite[Exer.\ 4.1.14]{Ro88}). Since $\alpha$ is bijective, $M^\alpha_W$
        is also f.g.\ projective, so we are done by Corollary \ref{GEN:CR:main-thm-base}(i).
        The form $b_\alpha$ is left regular by symmetry, as explained at the end of section \ref{section:GEN:preface}.
    \end{proof}

    It turn out that the assumptions of Theorem~\ref{GEN:TH:main-thm-II} also imply that every regular form is generic.
    This will be verified in Theorem \ref{GEN:TH:main-thm-IV} below.
    We also note that that it is essential to assume
    that $\alpha$ is bijective in Theorem \ref{GEN:TH:main-thm-II}; see Example \ref{GEN:EX:example-IV}.

    \smallskip
}

    The following example presents cases in which Theorem~\ref{GEN:TH:main-thm-II}
    can be applied to big families of right $R$-modules.

    \begin{example}
        (i) If $R$ is a Dedekind domain and $M_R$ is f.g.\ then $M$ is a generator when considered as an
        $R/\ann(M)$-module. (This follows from classification of f.g.\ modules over Dedekind domains;
        e.g.\ see \cite[Th.\ 4.14]{MaximalOrders}.) Therefore, the conclusions of Theorem~\ref{GEN:TH:main-thm-II}
        apply to $M$.

        (ii) A ring $R$ is called right \emph{pseudo-Frobenius} (abbrev.: PF) if any faithful right $R$-module is
        a generator. This is equivalent to $R$ being a right self-injective semilocal ring with essential right socle;
        see \cite[Th.\ 12.5.2]{Ka82ModulesAndRings} for other definitions.
        In this case, Theorem~\ref{GEN:TH:main-thm-II} applies to all faithful $R$-modules.
        Examples of two-sided PF rings include semisimple rings, artinian rings with a simple socle, Frobenius algebras,
        and finite group algebras over the previous examples;
        see \cite[Ch.\ 5]{La99}.
    \end{example}

\section{Orthogonal Sums}
\label{section:GEN:orthogonal-sums}

    In this section, we define orthogonal sums and prove a result about how they
    interact with the map $\alpha\mapsto b_\alpha$ of section~\ref{section:GEN:preface}. This will be used in the next section.

    \medskip

    Let $K$ be a double $R$-module and let $(M,b,K)$ and $(M',b',K)$ be bilinear
    spaces. The \emph{orthogonal sum} $(M,b,K)\perp(M',b',K)$ is defined to be
    $(M\oplus M',b\perp b',K)$ where
    \[
    (b\perp b')(x\oplus x',y\oplus y')=b(x,y)+b'(x',y')\qquad\forall~x,y\in M,~x',y'\in M'\ .
    \]
    Observe that when viewed as submodules of $M\oplus M'$, $M$ and $M'$ satisfy
    \[(b\perp b')(M,M')=(b\perp b')(M',M)=0\ .\]
    Conversely, if $(b'',M'',K)$ is a bilinear space
    and there are submodules $M,M'\leq M''$ such that $M''=M'\oplus M$ and $b''(M,M')=b''(M',M)=0$,
    then $b''=b\perp b'$ with $b=b''|_{M\times M}$ and $b'=b''|_{M'\times M'}$.

    \begin{prp}\label{GEN:PR:orth-sum-regularity}
        Let $(M_1,b_1,K)$, $(M_2,b_2,K)$ be bilinear spaces. Then $b_1\perp b_2$ is right regular (injective)
        $\iff$ $b_1$ and $b_2$ are right regular (injective).
    \end{prp}

    \begin{proof}
        Identify $(M_1\oplus M_2)^{[1]}$ with $M_1^{[1]}\oplus M_2^{[1]}$ in the standard way (i.e.\ via\linebreak $f\mapsto (f|_{M_1},f|_{M_2})$).
        Then it is straightforward to check that $\rAd{b_1\perp b_2}=\rAd{b_1}\oplus \rAd{b_2}$. The proposition follows immediately.
    \end{proof}

    Fix a right $R$-module $M$ and let $W=\End_R(M)$.
    Then we can identify $\End_R(M^n)$ with $\nMat{W}{n}$ for all $n\in\N$ (the elements of $M^n$ are considered
    as column vectors). For every $\alpha\in\aEnd(W)$, let $T_n\alpha$ denote the anti-automorphism of
    $\nMat{W}{n}$ defined by:
    \[\left[\DotsArr{w_{11}}{w_{1n}}{w_{n1}}{w_{nn}}\right]^{T_n\alpha}=
    \left[\DotsArr{w_{11}^\alpha}{w_{n1}^\alpha}{w_{1n}^\alpha}{w_{nn}^\alpha}\right]\]
    (the notation $T_n\alpha$ means ``apply $n\times n$ transpose and then $\alpha$ coordinate-wise'').
    In addition, for a bilinear form, $b:M\times M\to K$,
    let
    \[
    n\cdot b:=\underbrace{b\perp\dots\perp b}_{\textrm{$n$ times}}:M^n\times M^n\to K\ .
    \]
    The rest of this section is devoted to showing that $b_{T_n\alpha}$ is always similar to $n\cdot b_\alpha$.

    \begin{lem}\label{GEN:LM:full-idempotent-lemma}
        Let $N\in\rMod{W}$, $M\in\lMod{W}$ and let $e\in W$ be an idempotent
        satisfying $WeW=W$.
        Define a group homomorphism $\vphi: Ne\otimes_{eWe} eM\to N\otimes_{W} M$
        by $\vphi(x\otimes_{eWe} y)= x\otimes_W y$. Then
        $\vphi$ is an isomorphism.
    \end{lem}

    \begin{proof}
        Write $1_W=\sum_i u_iu'_i$ where $u_1,\dots,u_t\in We$ and $u'_1,\dots,u'_t\in eW$ and
        define $\psi: N\otimes_W M\to Ne\otimes_{eWe} eM$ by $\psi(x\otimes_W y)=\sum_{i}xu_i\otimes_{eWe}u'_iy$.
        Then $\psi$ is well defined because
        \begin{eqnarray*}
        \psi(xw\otimes_W y)&=&
        \sum_{i}xw u_i\otimes_{eWe}u'_iy=
        \sum_{i,j}xu_ju'_jwu_i\otimes_{eWe}u'_iy\\
        &=&
        \sum_{i,j}xu_j\otimes_{eWe}u'_jw u_i u'_iy=
        \sum_j xu_j\otimes_{eWe}u'_jwy=
        \psi(x\otimes_W wy)
        \end{eqnarray*}
        and it is straightforward to check that $\psi=\vphi^{-1}$.
    \end{proof}

    An idempotent $e\in W$ satisfying $WeW=W$ is called \emph{full}. This condition
    is equivalent to $eW_W$ (or ${}_WWe$) being a progenerator; see \cite[Rm.\ 18.10]{La99}. For example, the standard matrix
    unit $e_{11}$ is a full idempotent in $\nMat{W}{n}$.

    \begin{prp}\label{GEN:PR:power-of-a-generic-form}
        For all $\alpha\in\aEnd(W)$, we have $b_{T_n\alpha}\sim n\cdot b_\alpha$.
    \end{prp}

    \begin{proof}
        Let $\{e_{ij}\}$ be the standard matrix units of $U:=\nMat{W}{n}$, and let $\psi_i:M\to M^n$
        be the embedding of $M$ as the $i$-th component of $M^n$. Choose some $1\leq i\leq n$, and define
        $f:K_\alpha\to K_{T_n\alpha}$ by $f(x\otimes_\alpha y)=\psi_i x\otimes_{T_n\alpha}\psi_i y$ for all $x,y\in M$. It is easy
        to see that $f$ is well-defined. Furthermore, $f$ is independent of $i$ because
        \[
        \psi_i x\otimes_{T_n\alpha}\psi_i y=e_{ij}e_{ji}\psi_ix\otimes_{T_n\alpha} \psi_iy=e_{ji}\psi_ix\otimes_{T_n\alpha} e_{ij}^{T_n\alpha} \psi_ix=
        \psi_jx\otimes_{T_n\alpha} \psi_jy\ .
        \]
        This in turn implies that for all $x=(x_1,\dots,x_n),y=(y_1,\dots, y_n)\in M^n$, we have
        \begin{eqnarray*}
        f((n\cdot b_\alpha)(x,y))&=&
        \sum_{i=1}^n f(b_\alpha(x_i,y_i))=\sum_{i=1}^nb_{T_n\alpha}(\psi_i x_i, \psi_i y_i)
        \\
        &=&
        \sum_{i=1}^nb_{T_n\alpha}(e_{ii}x,e_{ii}y)=\sum_{i=1}^nb_{T_n\alpha}(x,e_{ii}^{T_n\alpha} e_{ii}y)=b_{T_n\alpha}(x,y)\ .
        \end{eqnarray*}
        Therefore, we are done if we show that $f$ is an isomorphism. To see this, identify $M$ with $e_{11}M^n$ and let
        $e=e_{11}$ and $\alpha'=T_n\alpha|_{eUe}$. Recall
        that $K_{T_n\alpha}=M^n\otimes_{T_n\alpha} M^n$ can be understood as $(M^n)^{T_n\alpha}\otimes_U (M^n)$ (see Remark~\ref{GEN:RM:twisting-remark}),
        and
        likewise, we can identify $K_\alpha=M\otimes_\alpha M=eM^n\otimes_{\alpha'}eM^n$ with
        $(eM^n)^{T_n\alpha}\otimes_{eUe}(eM^n)=(M^n)^{T_n\alpha}e\otimes_{eUe} e(M^n)$. Now, the map $f$ is just
        the map $\vphi$ of Lemma~\ref{GEN:LM:full-idempotent-lemma}, hence it is an isomorphism because $UeU=U$.
    \end{proof}

\section{The Structure of $K_\alpha$}
\label{section:GEN:generators}

    Let $R,M,W$ be as in section \ref{section:GEN:preface}.
    In this section, we use the results of the previous sections
    to obtain a (relatively) explicit description of
    $b_\alpha$ and $K_\alpha$ ($\alpha\in \aEnd(W)$)
    in case $M$ is a generator. We then use this description for several of applications.

    \begin{dfn}
        Let $M$ be a right $R$-module and $W=\End_R(M)$. The module $M$ is called \emph{faithfully balanced} if
        the standard map $R\to \End_W(M)$ is an isomorphism.
    \end{dfn}

    \begin{example}
        It is well known that any $R$-module which is generator is faithfully balanced; e.g., see \cite[Exer.\ 4.1.14]{Ro88}.
    \end{example}

    Let $M$ be a generator of $\rMod{R}$. Then
    $R_R$ is a summand of
    $M^n$ for some $n\in\N$. Let $e$ be the projection from
    $M^n$ onto the summand $R_R$.
    Then $e$ is an idempotent in
    $\End_R(M^n)$, which we identify with $U:=\nMat{\End_R(M)}{n}=\nMat{W}{n}$.
    Observe that ${}_UUe\cong {}_UM^n$ via $ue\mapsto u(1_R)$.
    (Here, $1_R$ is the unity of $R$, considered as an element of $M^n$. The inverse of
    this isomorphism
    is given by $x\mapsto [m\mapsto x\cdot (em)]\in U$.)
    Identify
    $Ue$ with $M^n$. Then, $\End({}_UM^n)= \End({}_UUe)=eUe$ and since $M^n_R$
    is faithfully balanced (it is a generator), it follows that $R\cong eUe$
    as rings, so we may assume $R=eUe$. In particular, $Ue$ and $M^n$ coincide as $(U,R)$-bimodules
    and $e_{ii}Ue$ is just the $i$-th copy of $M$ in $M^n$ (where $\{e_{ij}\}$ are the standard matrix units in $U$).

    \begin{prp}\label{GEN:PR:K-alpha-computation}
        Keeping the previous notation, let $\alpha\in\aEnd(W)$
        and let $\beta:=T_n\alpha\in\aEnd(U)$, with $T_n\alpha$ defined as in section~\ref{section:GEN:orthogonal-sums}.
        Make $e^{\beta} Ue$ into  a double
        $R$-module by letting
        \[u\mul{0} r=r^{\beta} u,\qquad u\mul{1} r=ur\qquad\forall\, r\in R=eUe,\, u\in e^{\beta} Ue\]
        and define $b:M\times M=e_{11}Ue\times e_{11}Ue\to e^{\beta} Ue$ by $b(x,y)=x^{\beta} y$. Then:
        \begin{enumerate}
        \item[(i)] $b_\alpha\sim b$. The similarity is given by $x\otimes_\alpha y\mapsto x^{\beta} y$ ($x,y\in M=e_{11}Ue$).
        \item[(ii)] Assume $\alpha$ is an involution. Then once identifying $K_\alpha$ with $eUe^{\beta}$ as in (i),
        $\theta_\alpha$ is just ${\beta}|_{e^{{\beta}}Ue}$ (see section~\ref{section:GEN:preface} for the definition of $\theta_\alpha$).
        \end{enumerate}
    \end{prp}

    \begin{proof}
        (i) Since $M$ is identified with $e_{11}Ue$, we may identify $\End_R(M)$
        with $e_{11}Ue_{11}$ and $\alpha$ with $\beta|_{e_{11}Ue_{11}}$.
        As in the proof of Proposition~\ref{GEN:PR:power-of-a-generic-form}, the map
        $M\otimes_\alpha M\to M^n\otimes_\beta M^n$ given by $x\otimes_\alpha y\mapsto x\otimes_\beta y$
        is an isomorphism.
        Identify $K_\beta$ with
        $(Ue)^\beta\otimes_U Ue$ as in Remark~\ref{GEN:RM:twisting-remark}.
        The latter is isomorphic to $(Ue)\twistmul{\beta} e=e^\beta Ue$
        via $x\otimes_U u\mapsto x\twistmul{\beta} u=u^\beta x$ (this is a general fact; for any
        $A\in\rMod{U}$, $A\otimes_UUe\cong Ae$). Part (i)
        now follows by composing the isomorphisms $K_\alpha\to K_\beta$ and $K_\beta\to e^\beta Ue$.
        Here is the explicit computation:
        \[
        \begin{array}{ccccccl}
        e_{11}Ue\otimes_{\alpha}e_{11}Ue &\cong&   Ue\otimes_\beta Ue &\cong& (Ue)^\beta\otimes_U Ue &\cong& e^\beta Ue \\
        x\otimes_{\alpha} y              &\mapsto& x\otimes_\beta y   &\mapsto& y\otimes_U x &\mapsto& y\twistmul{\beta}x=x^\beta y\ .
        \end{array}
        \]

        (ii) Assume $\alpha$ is an involution and identify $K_\alpha$ with $e^\beta Ue$.
        Then for all $x,y\in e_{11}Ue$, $(x\otimes_{\alpha} y)^{\theta_\alpha}=y\otimes_\alpha x$,
        so under the identification $K_\alpha\cong e^\beta U e$ we get
        $(x^\beta y)^{\theta_\alpha}=y^\beta x$ and the latter equals $(x^\beta y)^\beta$ since $\beta$
        is also an involution. Thus, $\theta_\alpha$ coincides with $\beta$ on $e^\beta Ue$.
    \end{proof}

    \begin{remark}
        In the proposition's assumptions,
        it also possible to understand $e^\beta Ue$ as $e^\beta M^n=\im(e^\beta)$. Under this identification,
        the form $b$ is given by
        the formula
        \[b(x,y)=([z\mapsto x\cdot ez]^\beta) (y)\qquad\forall x,y\in M\ .\]
        Here, $z\in M^n$ and
        $M$ has to identified as one of the summands of $M^n=M\oplus\dots\oplus M$.
    \end{remark}

    Proposition~\ref{GEN:PR:K-alpha-computation} enables us to give another  proof for Theorem
    \ref{GEN:TH:main-thm-II} and even strengthen it.

    \begin{cor}\label{GEN:CR:generic-forms-over-generators}
        Assume $M\in\rMod{R}$ is generator, let $W=\End(M_R)$ and let $\alpha\in\aEnd(W)$.
        If $\alpha$ is injective, then $b_\alpha$ is left injective. If $\alpha$ is bijective,
        then $b_\alpha$ is (right and left) regular.
    \end{cor}

    \begin{proof}
        By Proposition~\ref{GEN:PR:power-of-a-generic-form} and Proposition \ref{GEN:PR:orth-sum-regularity},
        we can replace $b_\alpha$ with $n\cdot b_\alpha$, thus assuming $n=1$, $U=\nMat{W}{1}=W$, $e_{11}=1$
        and $\beta=\alpha$ in previous computations.
        (This step is not really necessary, but it simplifies the arguments to follow.)
        Define $b$  as in Proposition \ref{GEN:PR:K-alpha-computation}. Then it is enough to prove that $b$ is injective/regular.
        Indeed,
        $b(x,M)=0$ implies $x^\alpha\in\annl Ue$. Since $U= \End_R(M)=\End_{eUe}(Ue)$, ${}_UUe$ is faithful,
        so $x^\alpha=0$.
        Thus, if $\alpha$ is injective, $x=0$, hence $b_\alpha$ is left injective.

        Now assume $\alpha$ is bijective. We claim that $\lAd{b}$ is surjective. This
        is easily seen to be equivalent to
        showing that
        that any $f\in\Hom_R(Ue,e^\beta Ue)$ is induced
        by left multiplication with an element of $(Ue)^\beta=e^\beta U$.
        Indeed, viewing $f$ as an endomorphism of $Ue$, we see that $f(x)=ux$
        for some \emph{unique} $u\in U$ (because $U=\End(M^n_R)=\End(Ue_R)$). Since $u$ and $e^\beta u$
        clearly induce the same endomorphism on $Ue_R$, we must have $u=e^\beta u$, as required.
        Thus, $\lAd{b}$ is surjective, hence  bijective by the previous paragraph.
        That $b$ right regular follows by symmetry.
    \end{proof}

    We could neither find nor contradict the existence of a generator $M$ with an \emph{injective} $\alpha\in\aEnd(W)$
    such that
    $b_\alpha$ is not right injective. However, if $\alpha$ is not injective, then it is possible that $b_\alpha$
    would be the zero form even when $M$ is a generator; see Example~\ref{GEN:EX:example-V}.

\medskip

    The rest of this section uses Proposition~\ref{GEN:PR:K-alpha-computation} to obtain
    various structural results about $K_\alpha$, provided certain assumptions on $M$ and $R$.
    In particular, we prove
    the claims
    posed in Example~\ref{GEN:EX:classical-involutions}.

    \begin{prp}\label{GEN:CR:dimesion-of-K-alpha-cor}
        Assume $M\in\rMod{R}$ is free of rank $n\in\N$,  let $W=\End(M_R)$ and let $\alpha\in\aEnd(W)$.
        Then $(K_\alpha)_1^{n}\cong R^{n}$ as right $R$-modules.
        (Recall that $(K_\alpha)_1$ means ``$K_\alpha$, considered as a right $R$-module
        w.r.t.\ $\mul{1}$''.)
    \end{prp}

    \begin{proof}
        Assume $M=R^n$ and identify $W$ with $\nMat{R}{n}$. Let $\{e_{ij}\}$ be the standard
        matrix units of $W$. Then by Proposition~\ref{GEN:PR:K-alpha-computation}, $K_\alpha\cong e_{11}^\alpha W e_{11}$
        (take $e=e_{11}$). Consider $K_i:=e_{ii}^\alpha W e_{11}$ as a \emph{right $R$-module}. Then $K_i\cong K_j$
        for all $i,j$ (the isomorphism being multiplication on the left by $e_{ij}^\alpha$).
        Thus, $(K_\alpha)_1^n\cong K_1\oplus \dots\oplus K_n=(\sum_i e_{ii}^\alpha) We_{11}=We_{11}\cong R_R^n$
        as right $R$-modules.
    \end{proof}

    \begin{lem}\label{GEN:LM:dfn-of-phi}
        Fix a double $R$-module $K$. For $M\in\rMod{R}$, define $\Phi_M:M\to M^{[1][0]}$ by
        $(\Phi_Mx)f=f(x)$ for all $x\in M$ and $f\in M^{[1]}$ (see section~\ref{section:FORM:definitions} for
        the definitions of $[0]$ and $[1]$). Then:
        \begin{enumerate}
            \item[(i)] $\{\Phi_M\}_{M\in\rMod{R}}$ is a \emph{natural transformation} from
            $\id_{\rMod{R}}$ to $[0][1]$ (i.e.\ for all $N,N'\in\rMod{R}$ and $f\in \Hom_R(N,N')$, one has
            $f^{[1][0]}\circ\Phi_N=\Phi_{N'}\circ f$).
            \item[(ii)] $\Phi$ is additive (i.e.\ $\Phi_{N\oplus N'}=\Phi_N\oplus\Phi_{N'}$ for all $N,N'\in\rMod{R}$).
            \item[(iii)] $(\rAd{b})^{[0]}\circ\Phi_M=\lAd{b}$ for every general bilinear form $b:M\times M\to K$.
            \item[(iv)] $R^{[1][0]}$ can be identified with $\End_R(K_1)$. Under that identification,
            $(\Phi_Rr)k=k\mul{0} r$ for all $r\in R$ and $k\in K$.
        \end{enumerate}
    \end{lem}

    \begin{proof}
        (i)--(iii) are  straightforward computation.
        (Recall that for $f\in\Hom_R(N,N')$, we have $f^{[i]}(\vphi)=\vphi\circ f$ for all $\vphi\in N^{[i]}$.)
        For example, $(\rAd{b})^{[0]}\circ\Phi_M=\lAd{b}$ holds since for all $x,y\in M$, we have
        $(\lAd{b}x)y=b(x,y)=(\rAd{b}y)x=(\Phi_Mx)(\rAd{b}y)=((\rAd{b})^{[0]}(\Phi_Mx))y$.

        (iv) We have $R^{[1]}=\Hom_R(R_R,K_0)\cong K_1$ via $f\mapsto f(1)$ ($f\in R^{[1]}$)
        and hence $R^{[1][0]}\cong K_1^{[0]}=\Hom_R(K_1,K_1)=\End_R(K_1)$. It is now routine to verify
        that under that isomorphism, the map $\Phi_R(r)$ is just $[k\mapsto k\mul{0}r]\in\End_R(K_1)$.
    \end{proof}

    \begin{thm}\label{GEN:TH:when-K-alpha-is-standard}
        Let $n\in\N$, $M\in\rMod{R}$, $W=\End(M_R)$ and $\alpha\in\aEnd(W)$. Assume that $M^k\cong R^n$ for some $k\in\N$
        and $N^n\cong R^n$ implies $N\cong R_R$ for all $N\in \rMod{R}$ (e.g.\ if $R$ is semilocal or a principal
        ideal domain). Then:
        \begin{enumerate}
            \item[(i)] There exists $\gamma=\gamma(\alpha)\in\aEnd(R)$ such that
            $K_\alpha$ is isomorphic to the standard double $R$-module of $(R,\gamma)$ (see Example~\ref{FORM:EX:base-example}).
            The anti-endomorphism $\gamma$ is unique up to composition with an inner automorphism of $R$.
            \item[(ii)] $\alpha\in\aAut(W)$ $\iff$ $\gamma\in\aAut(R)$.
            \item[(iii)] There exists $\lmb\in W$ with $\lmb^\alpha\lmb=1$ (resp.\ $\lmb^\alpha\lmb\in\units{W}$)
            and $w^{\alpha\alpha}\lmb=\lmb w$ for all $w\in W$ $\iff$
            there exists $\mu\in R$ with $\mu^\gamma\mu=1$ (resp.\ $\mu^\gamma\mu\in\units{R}$) and $r^{\gamma\gamma}\mu=\mu r$
            for all $r\in R$.\footnote{
                The elements $\lmb$, $\mu$ are invertible when $\alpha$, $\gamma$ are invertible, respectively; see Remark~\ref{GEN:RM:lmb-is-not-invertible}.
            }
            \item[(v)] The map $\alpha\mapsto \gamma(\alpha)$ gives rise to injective maps
            \begin{align*}
            \Inn(W)\setminus\!\aEnd(W)& \to  \Inn(R)\setminus\!\aEnd(R),\\
            \Inn(W)\setminus\!\aAut(W)& \to  \Inn(R)\setminus\!\aAut(R)\ .
            \end{align*}
            If $M\cong R^n$, then these maps are bijective.
        \end{enumerate}
    \end{thm}

    \begin{proof}
        By Proposition~\ref{GEN:PR:power-of-a-generic-form}, we may replace $M$ with $M^k$ and henceforth assume $M=R^n$.
        Throughout, $S_\gamma$ denotes the standard double $R$-module of $(R,\gamma)$. Recall that $S_\gamma=R$ as sets
        and $k\mul{0}r=r^\gamma k$ and $k\mul{1}r=kr$ for all $k,r\in R$.

        (i)
        By Proposition~\ref{GEN:CR:dimesion-of-K-alpha-cor}, $(K_\alpha)_1^n\cong R^n$, so by assumption, $(K_\alpha)_1\cong R_R$.
        Choose a basis $\{k\}$ to  $(K_\alpha)_1$. Then for all $r\in R$, there is a unique $r^\gamma\in R$ such that
        $k\mul{0} r=k\mul{1} r^\gamma$. The map $\gamma$ is easily seen to be an anti-automorphism of $R$, and
        it is routine to verify that $S_\gamma\cong K_\alpha$ via  $r\mapsto k\mul{1}r$.

        To see that $\gamma$ is unique up to composition with an inner-automorphism, it enough to
        show that for all $\gamma,\gamma'\in\aEnd(R)$,
        $S_\gamma\cong S_{\gamma'}$ $\iff$ there exists $u\in \units{R}$
        such that $r^\gamma=u^{-1}r^{\gamma'}u$ for all $r\in R$.
        Indeed, since $(S_\gamma)_1=(S_{\gamma'})_1=R_R$, we have $S_\gamma\cong S_{\gamma'}$ $\iff$
        there exists $u\in\units{R}$ such that
        $u(r^\gamma k)=r^{\gamma'}(uk)$ for all $r,k\in R$.
        But $u(r^\gamma k)=r^{\gamma'}(uk)$ for all $r,k\in R$
        $\iff $$ur^\gamma=r^{\gamma'}u$ for all $r\in R$ $\iff$ $r^\gamma=u^{-1}r^{\gamma'}u$ for all $r\in R$, so we are done.

        (ii)
        Identify $K:=K_\alpha$ with $S_\gamma$, let $b=b_\alpha$, and let $\Phi_M$ be as in Lemma~\ref{GEN:LM:dfn-of-phi}.
        By that lemma, we have $(\rAd{b})^{[0]}\circ \Phi_M=\lAd{b}$. Since $\rAd{b}$ is bijective (Theorem~\ref{GEN:TH:main-thm-II}),
        $\lAd{b}$ is bijective if and only if $\Phi_M$ is bijective, and since $M=R^n$ and $\Phi$ is additive, the latter is equivalent
        to $\Phi_R$ being bijective. Identifying $R^{[1][0]}$ with $\End_R((S_\gamma)_1)=\End_R(R_R)\cong R$ as in
        Lemma~\ref{GEN:LM:dfn-of-phi}(iv), we see that $\Phi_R$ is just the map $\gamma$. Therefore, $\lAd{b}$ is bijective
        if and only if $\gamma$ is.

        Now, if $\gamma$ is bijective, then $b$ is right and left regular, implying $\alpha$ is invertible (its inverse
        is the \emph{left} corresponding anti-endomorphism of $b$, which exists since $b$ is left regular). On the other hand,
        if $\alpha$ is invertible, then $b$ is left regular by Theorem~\ref{GEN:TH:main-thm-II}, implying $\gamma\in \aAut(R)$.

        (iii)
        Define $b':R\times R\to S_\gamma$ by $b'(x,y)=x^\gamma y$. By Example~\ref{FORM:EX:base-example}, $c$ is right
        regular with corresponding anti-endomorphism $\gamma$, so by Theorem~\ref{GEN:TH:main-thm-II}, $b'\sim b_\gamma$.
        Now, by Proposition~\ref{GEN:PR:basic-properties-of-b-alpha-II}, the existence of $\mu$ as above is equivalent
        to the existence of an involution (resp.\ anti-automorphism) on $S_\gamma$. However,
        the same proposition implies that the existence of
        $\lmb$ as above is equivalent to the existence of an involution (resp.\ anti-automorphism)
        on $K_\alpha$. Since $K_\alpha\cong S_\gamma$, we are done.

        (iv)
        Let $\alpha,\alpha'\in\aEnd(W)$ and assume $\Inn(R)\gamma(\alpha)=\Inn(R)\gamma(\alpha')$.
        By the proof of (i), this means $K_\alpha\cong K_{\alpha'}$, so by Proposition~\ref{GEN:PR:K-alpha-isomorphic-to-K-beta},
        $\Inn(W)\alpha=\Inn(W)\alpha'$, as required.

        To finish, assume $M\cong R^n$ and let $\gamma\in\aEnd(R)$. Define $b_\gamma$ as
        in (iii). Then by Proposition~\ref{GEN:PR:power-of-a-generic-form}, $n\cdot b_\gamma\sim b_{T_n\gamma}$.
        Since $M\cong R^n$, we may view $b_{T_n\gamma}$ as a form on $M$,
        and hence assume $T_n\gamma\in W$. But now we have $K_{T_n\gamma}\cong K_\gamma\cong S_\gamma$,
        so  by definition, $\gamma(T_n\gamma)=\gamma$.
    \end{proof}

    As a special case of the theorem, we get that for every ring $R$ such that $N^n\cong R^n$ implies $N\cong R_R$ ($N\in\rMod{R}$; $n$ fixed),
    we have a set bijection
    \begin{equation}\label{GEN:EQ:anti-auto-of-matrices}
        \begin{array}{ccc}
        \Inn(R)\setminus\!\aEnd(R) &\quad\cong\quad&  \Inn(\nMat{R}{n})\setminus\!\aEnd(\nMat{R}{n}) \\
        \Inn(R)\cdot \gamma &\mapsto&  \Inn(\nMat{R}{n})\cdot T_n\gamma
        \end{array}
    \end{equation}
    which maps $\Inn(\nMat{R}{n})\setminus\!\aAut(\nMat{R}{n})$ to  $\Inn(R)\setminus\!\aAut(R)$.
    In particular,
    $R$ has an anti-automorphism (resp.\ anti-endomorphism) $\iff$ $\nMat{R}{n}$ has an anti-automorphism (resp.\ anti-endomorphism).
    In case $\aAut(R)\neq\phi$ (e.g.\ if $R$ is commutative), we also get a group isomorphism
    \[
    \Out(R)\cong \Out(\nMat{R}{n})
    \]
    where $\Out(R)$ is the \emph{outer automorphism group} of $R$, namely $\Aut(R)/\Inn(R)$.
    Indeed, let $W=\nMat{R}{n}$ and fix some $\gamma_0\in \aAut(R)$.
    Observe that $\Out(R)\cong \Inn(R)\setminus\!\aAut(R)$ as sets via $\Inn(R)\vphi\mapsto \Inn(R)(\vphi\circ\gamma_0)$
    and likewise, $\Out(W)\cong \Inn(W)\setminus\!\aAut(W)$ via $\Inn(W)\psi\mapsto \Inn(W)(\psi\circ T_n\gamma_0)$. Composing
    these isomorphisms with the bijection in \eqref{GEN:EQ:anti-auto-of-matrices}
    gives a set bijection $\Out(R)\to\Out(W)$ sending the coset
    of $\vphi\in \Aut(R)$ to the coset of the automorphism $[(r_{ij})_{i,j}\mapsto (\vphi(r_{ij}))_{i,j}]\in \Aut(W)$, but
    this bijection is also a group homomorphism.

    \begin{remark}
        (i)
        In general, even when $M$ is a progenerator, very little can be said about the structure of $K_\alpha$.
        For example, the base module $M$ of the regular bilinear
        form $b:M\times M\to K$ constructed in Example~\ref{GEN:EX:form-over-incidence-algebra} is a progenerator.
        Since $b$ is regular (routine), it is similar to $b_\alpha$ for some  $\alpha\in\aAut(\End_R(M))$ (Theorem~\ref{GEN:TH:main-thm-II}),
        what allows
        us to identify $K$ with $K_\alpha$.
        However, $(K_1)^m$ is not free for all $m\in \N$, in contrast to Proposition~\ref{GEN:CR:dimesion-of-K-alpha-cor}
        and Theorem~\ref{GEN:TH:when-K-alpha-is-standard}(i).

\rem{
        (ii) Theorem~\ref{GEN:TH:when-K-alpha-is-standard} and
        Corollary~\ref{GEN:CR:anti-automorphism-transfer} fail completely if we drop the assumption that $N^n\cong R^n$ implies $N\cong R_R$;
        Example~\ref{GEN:EX:last-example} below
        presents a noetherian ring $S$ such that $\aAut(S)=\phi$ but $\aAut(\nMat{S}{4})\neq \phi$.
}

        (ii) It is possible to prove Theorem~\ref{GEN:TH:when-K-alpha-is-standard} without using the form
        $b_\alpha$ explicitly:
        Identify $M=R^n$ with $\nMat{R}{n}e_{11}$ and define
        $b:M\times M\to K:=e_{11}^\alpha\nMat{R}{n}e_{11}$ by $b(xe_{11},ye_{11})=e_{11}^\alpha x^\alpha ye_{11}$.
        Prove directly that $b$ is right regular and $(K_1)^n\cong R^n$, and proceed with $b$, $K$ in place of $b_\alpha$, $K_\alpha$.

        (iii) Let $\alpha\in\aEnd(W)$ with $W$ as in Theorem~\ref{GEN:TH:when-K-alpha-is-standard}.
        If there is $\vphi\in\Inn(W)$ such that $\vphi\circ\alpha$ is an involution, then
        there exists $\lmb\in W$ with $\lmb^\alpha\lmb=1$
        and $w^{\alpha\alpha}\lmb=\lmb w$ for all $w\in W$ (write $\vphi(w)=uwu^{-1}$ and take $\lmb=u^{-1}u^\alpha$). The converse
        is false in general, but it is true when $W$ is semisimple  (see \cite[Pr.\ 2.1]{Reiter75}, for instance).
        Together with Theorem~\ref{GEN:TH:when-K-alpha-is-standard}, this can be used to show that when $R$ is simple artinian, all
        involutions on $W=\End_R(M)$ are induced by regular \emph{$\mu$-hermitian} forms over $R$ ($\mu\in\Cent(R)$). The details
        are left to the reader.
    \end{remark}

    We finish this section by specializing to the case where $R$ is a field.

    \begin{prp}\label{GEN:PR:classical-cor-over-fields}
        Let $F$ be a field and let $V$ be a f.d.\ right $F$-vector-space. Let $W=\End_F(V)$
        and identify $F$ with $\Cent(W)$. For all $\alpha\in \aEnd(W)$ we have:
        \begin{enumerate}
            \item[(i)] $\alpha(\Cent(W))\subseteq\Cent(W)$, hence $\alpha$ can be understood as an (anti-)endo\-mor\-phism
            of $F$. In addition, $\alpha|_F$ is bijective precisely when $\alpha$ is.
            \item[(ii)] $K_\alpha$ is isomorphic to the standard double $F$-module of $(F,\alpha|_F)$.
            \item[(iii)]
            There is a one-to-one correspondence between \emph{right regular}
            bilinear forms on $V$, considered up to similarity, and anti-endomorphisms of $W$.
            More precisely, the form $b_\alpha$ is right regular and left injective. It is left regular if and only if $\alpha$
            is bijective.
            \item[(iv)] If $(\alpha|_F)^2=\id_F$, then $b_\alpha$ is similar to a classical regular bilinear form
            (i.e.\ a sesquilinear form) over $(F,\alpha|_F)$. That bilinear form is unique up to multiplying by a scalar (in $\units{F}$).
            Conversely, if $\beta$ is an involution of $F$ and $b_\alpha$ is similar to  a classical
            regular bilinear form over $(F,\beta)$, then $\alpha|_F=\beta$ and in particular, $(\alpha|_F)^2=\id_F$.
            \item[(v)] If $\alpha$ is an involution,
            then one of the following is true:
            \begin{enumerate}
                \item[(1)] $\alpha|_F=\id_F$ and $b_\alpha$ is similar to a symmetric bilinear form over $F$.
                \item[(2)] $\alpha|_F=\id_F$ and $b_\alpha$ is similar to an alternating bilinear form over $F$.
                \item[(3)] $\alpha|_F\neq\id_F$ and $b_\alpha$ is similar to a $1$-hermitian form over $(F,\alpha|_F)$.
            \end{enumerate}
            If moreover $\Char F\neq 2$, then \emph{precisely} one of the above is true. (By definition, the cases (1),(2),(3) correspond
            to the cases where $\alpha$ is orthogonal, symplectic or unitary, respectively.)
        \end{enumerate}
    \end{prp}

    \begin{proof}
        (i) Since $W$ is simple, $\ker\alpha=0$. Fix some matrix units $\{e_{ij}\}$ in $W$. Then $\{e_{ij}^\alpha\}$
        are also matrix units for $W$ and it is easy to see that every element that commutes with $\{e_{ij}^\alpha\}$ lies in $\Cent(W)$.
        As  $\alpha(\Cent(W))$ commutes with $\{e_{ij}^\alpha\}$, we get that $\alpha(\Cent(W))\subseteq\Cent(W)$.
        If $\alpha$ is bijective, then  by the same argument, $\alpha^{-1}(\Cent(W))\subseteq\Cent(W)$, which
        implies $(\alpha|_{F})^{-1}=\alpha^{-1}|_F$. On the other hand, if $\alpha|_F$ is bijective, then $\alpha(W)$
        is an $F$-subalgebra of $W$ (since $\alpha(F)=F$) of dimension $(\dim V)^2$ (since
        $\ker \alpha=0$). Thus,  $\alpha$ is surjective, and hence bijective.

        (ii)
        By Theorem~\ref{GEN:TH:when-K-alpha-is-standard}, $K_\alpha$ is isomorphic to the standard
        double $F$-module of $(F,\gamma)$ for some $\gamma\in\aEnd(F)$. (In
        fact, $\gamma$ is unique since $\Inn(F)=\{\id_F\}$.)  For all $a\in F=\Cent(W)$ and $x,y\in V$, we have
        $b_\alpha(x,y)\mul{0} a=b_\alpha(a x,y)=b_\alpha(x,a^\alpha y)=b_\alpha(x,y)\mul{1} a^\alpha$. This forces $\gamma=\alpha|_F$, as
        required.

        (iii)
        Observe that any $\alpha\in\aEnd(W)$ is injective because $W$ is simple. Everything now follows from
        Theorem~\ref{GEN:TH:main-thm-II} and Corollary~\ref{GEN:CR:generic-forms-over-generators}.

        (iv)
        Assume $(\alpha|_F)^2=\id_F$. Then $\alpha|_F$ is bijective, hence $\alpha$ is bijective (by (i)).
        Thus,
        $b_\alpha$ is regular (by (iii)) and $K_\alpha$ is the standard
        double $F$-module of $(F,\alpha|_F)$, which, according to Example \ref{GEN:EX:standard-classical-bilinear-forms},
        means that $b_\alpha$ is similar to a classical bilinear form $b:V\times V\to F$.
        That $b$ is unique to multiplying by an element of $\units{F}$ easily follows from the fact that $b$
        is unique up to similarity.
        Conversely, if $b_\alpha$ is similar to a regular hermitian form $b:V\times V\to F$ over $(F,\beta)$,
        then for all $a\in F$ and $x,y\in V$, $b(xa,y)=a^\beta b(x,y)=b(x,ya^\beta)$. Thus, $\alpha|_F=\beta$, as required.

        (v) That $\alpha$ is an involution implies $(\alpha|_F)^2=\id_F$, hence by (iv), $b_\alpha$ is similar
        to a classical bilinear form $b:V\times V\to F$ over $(F,\alpha|_F)$.
        Identify $K_\alpha$ with $F$ via the similarity $b_\alpha\sim b$ and consider $\theta:=\theta_\alpha$ (see
        section~\ref{section:GEN:preface}) as an involution of the standard double $F$-module
        of $(F,\alpha|_F)$ (rather than $K_\alpha$). Let $\lmb=1_F^{\theta}$. Then for all $k\in F$,
        $k^{\theta}=(1\mul{1} k)^{\theta}=1^{\theta}\mul{0} k=\lmb k^\alpha$
        and $1=1^{\theta\theta}=\lmb^{\theta}=\lmb\lmb^{\alpha}$. As $b$ is  $\theta$-symmetric, this means that $b$ is $\lmb$-hermitian.
        Now, if $\alpha|_F=\id_F$, then $\lmb^2=1$, hence $\lmb=1$ or $\lmb=-1$ which implies (1) or (2), respectively.
        If $\alpha|_F\neq\id_F$, then by Hilbert's Theorem 90, there is $u\in F$ with $u^\alpha u^{-1}=\lmb$.
        Define $b'(x,y)=ub(x,y)$. Then $b'$ is similar to $b$ and it is routine to check that $b'$ is a $1$-hermitian form
        over $(F,\alpha|_F)$, i.e.\ (3) holds.
    \end{proof}

    \begin{remark}
        The results of the Proposition \ref{GEN:PR:classical-cor-over-fields} are not typical;
        in general, even when $M$ is free and $R$ has an involution, there is no ``nice'' characterization
        of the anti-endomorphisms of $W$ that correspond to \emph{classical} bilinear forms.
    \end{remark}

\section{On a Result of Osborn}
\label{section:GEN:osborn}

    In this section, we use the results of sections \ref{section:GEN:regularity}
    and \ref{section:GEN:generators} to prove  a variant of a theorem of Osborn. Throughout, $\Jac(R)$ is
    the Jacobson radical of the ring $R$.

    \begin{thm}[Osborn]
        Let $(W,\alpha)$ be a  ring with involution such that $2\in \units{W}$ and every element
        $w\in W$ with $w^\alpha=w$ is either a unit or nilpotent.
        Let $\alpha'$ denote the induced
        involution on $W/\Jac(W)$. Then $\Jac(W)\cap \{w\in W\suchthat w^\alpha=w\}$
        consists of nilpotent elements and one of the following holds:
        \begin{enumerate}
            \item[(i)] $W/\Jac(W)$ is a division ring.
            \item[(ii)] $W/\Jac(W)\cong D\times D^\op$ for some division ring $D$ and under that isomorphism
            $\alpha'$ exchanges $D$ and $D^\op$.
            \item[(iii)] $W/\Jac(W)\cong \nMat{F}{2}$ for some field $F$ and under that isomorphism
            $\alpha'$ is a symplectic involution (i.e.\ it is induced by a classical alternating bilinear form over $F$; see \cite[Ch.\ 1]{InvBook}).
        \end{enumerate}
    \end{thm}

    \begin{proof}
        See \cite[\S4]{Osborn70}.
    \end{proof}

    Osborn's result has several generalizations (see papers related to \cite{Osborn70}) and his
    proof is based on Jordan algebras.
    We will prove Osborn's Theorem in the case $W$ is semisimple, but under milder assumptions.
    Our proof actually implies Osborn's Theorem for semilocal rings, for one can easily reduce to the case
    $\Jac(W)=0$; see \cite[\S4]{Osborn70}. Our techniques are very different, though, and will rely on
    Theorem~\ref{GEN:TH:main-thm-II}.

    \begin{thm}\label{GEN:TH:gen-of-Osborns-result}
        Let $(W,\alpha)$ be a  ring with involution such that $W$ is semisimple and the only
        $\alpha$-invariant\footnote{
            An element $w\in W$ is \emph{$\alpha$-invariant} if $w^\alpha=w$. Some texts use \emph{$\alpha$-symmetric} instead of $\alpha$-invariant.
        } idempotents in $W$ are $0$ and $1$. Then one of the following holds:
        \begin{enumerate}
            \item[(i)] $W$ is a division ring.
            \item[(ii)] $W\cong D\times D^\op$ for some division ring $D$ and under that isomorphism
            $\alpha$ exchanges $D$ and $D^\op$.
            \item[(iii)] $W\cong \nMat{F}{2}$ for some field $F$ and under that isomorphism
            $\alpha$ is a symplectic involution.
        \end{enumerate}
    \end{thm}

    \begin{proof}
        We may assume $W$ is not the zero ring.
        Let $\{e_1,\dots,e_n\}$ be the primitive idempotents of $\Cent(W)$.
        Then $\alpha$ permutes $e_1,\dots,e_n$. Assume $n>1$. Then we have $e_i\neq e_i^\alpha$
        for all $i$. This implies $e_1+e_1^\alpha$ is a nonzero $\alpha$-invariant idempotent, hence
        $e_1+e_1^\alpha=1$. Thus, $n=2$ and $e_1^*=e_2$. Write $W_i=e_iW$. Then $W\cong W_1\times W_2$
        and $\alpha$ exchanges $W_1$ and $W_2$. If $0\neq e\in W_1$ is an idempotent, then $e^\alpha\in W_2$,
        hence $e+e^\alpha$ is a non-zero $\alpha$-invariant idempotent, hence $e+e^\alpha=1$ and $e=1_{W_1}$. This means
        $W_1$ is a simple artinian ring with no non-trivial idempotents, hence it is a division ring. As
        $W_2\cong W_1^\op$ via $\alpha$, (ii) holds.

        Now assume $n=1$. Then $W$ is simple artinian and we can write $W=\End_D(V)$ for some division
        ring $D$ and a f.d.\ right $D$-vector-space $V$. Let $b=b_\alpha$, $K=K_\alpha$ and $\theta=\theta_\alpha$.
        Then $b:V\times V\to K$ is a regular $\theta$-symmetric bilinear form by Theorem~\ref{GEN:TH:main-thm-II}.
        Moreover, by Proposition~\ref{GEN:CR:dimesion-of-K-alpha-cor}, $\dim_D K_1=1$.

        We claim that if $V=U_1\oplus U_2$ with $b(U_1,U_2)=b(U_2,U_1)=0$, then $U_1=0$ or $U_2=0$.
        Indeed, let $e\in\End_D(V)$ be the projection onto
        $U_1$ with kernel $U_2$. It is straightforward to check that $b(ex,y)=b(ex,ey)=b(x,ey)$,
        hence $e^\alpha=e$. Therefore, $e=1$  or $e=0$, so $U_1=V$ or $U_1=0$.

        Assume there is $x\in V$ such that $b(x,x)\neq 0$ and
        define
        \[L=x^\perp=\{y\in V\where b(x,y)=0\}\ .\]
        We claim that $V=L\oplus xD$. Clearly $x\notin L$, hence $xD\cap L=0$.
        On the other hand, for all $v\in V$, there exists
        $d\in D$ such that $b(x,x)\mul{1}d=b(x,v)$ (because $\dim_D K_1=1$),
        hence $b(x,v-xd)=b(x,v)-b(x,x)\mul{1}d=0$ and this implies $v=xd+(v-xd)\in xD+L$.
        Now,  since $b$ is $\theta$-symmetric $b(L,xD)=b(xD,L)^\theta=0$, so by the previous paragraph,
        $L=0$.
        But this means $\dim_D V=1$, so $W\cong D$ and (i) holds.

        We may now assume that $b(x,x)=0$ for all $x\in V$.
        Then  $0=b(x+y,x+y)=b(x,y)+b(y,x)=b(x,y)+b(x,y)^\theta$ for all $x,y\in V$,
        hence $\theta=-\id_K$.
        Furthermore, for all $x,y\in V$ and
        $a\in D$ we have $b(x,y)\mul{0}a=b(xa,y)=-b(y,xa)=-b(y,x)\mul{1}a=b(x,y)\mul{1}a$, hence $\mul{0}=\mul{1}$.
        This implies that for any $0\neq k\in K$ and $a,b\in D$, we have $k\mul{1}(ab)=(k\mul{1}a)\mul{1}b=
        (k\mul{0}a)\mul{1}b=(k\mul{1}b)\mul{0}a=(k\mul{1}b)\mul{1}a=k\mul{1}(ba)$, hence $ab=ba$.
        Therefore, $D$ is a field and $K$ is isomorphic as a double $D$-module to
        the standard double module of $(D,\id_D)$. As $b(x,x)=0$ for all $x\in V$,
        $b$ is a classical alternating bilinear form.
        We are thus finished if we prove that $\dim_DV=2$ (as this would
        imply $W\cong\nMat{D}{2}$, as in (iii)). However, this follows from the well-known fact
        that every regular alternating form over a field is the orthogonal sum of $2$-dimensional alternating forms
        (and $b$ cannot be the orthogonal sum of two non-trivial forms, as argued above).
    \end{proof}

    Theorem~\ref{GEN:TH:gen-of-Osborns-result} is false for rings $W$ with $\Jac(W)=0$ (and hence
    it does not imply Osborn's Theorem for general rings). For example, for any simple \emph{domain} $S$ which
    is not a division ring (e.g.\ a Weyl algebra),
    the ring $W=S\times S^\op$ has an involution and satisfies $\Jac(W)=0$,
    but it fails to satisfy any of the conditions (i)--(iii) of Theorem~\ref{GEN:TH:gen-of-Osborns-result}.

\section{Counterexamples}
\label{section:GEN:examples}

    This last section presents counterexamples.
    We begin with  demonstrating that $b_\alpha$ can be degenerate even
    when the base ring is a finite dimensional algebra over a field.

    \begin{example}
        Let $F$ be a field and let $R$
        be the commutative subring of $\nMat{F}{3}$ consisting of matrices of the form:
        \[\SMatIII{a}{}{}{b}{a}{}{c}{}{a}\ .\]
        Let $x=e_{21}$ and $y=e_{31}$ (where $\{e_{ij}\}$ are
        the standard matrix units of $\nMat{F}{3}$). Then
        $\{1,x,y\}$ is an $F$-basis of $R$. Consider the elements of $M=F^3$ as \emph{row vectors}
        and let $\{e_1,e_2,e_3\}$ be the standard $F$-basis of $M$. Then $M$ is naturally
        a right $R$-module (the action of $R$ being matrix multiplication on the right) and a straightforward
        computation shows that  $\End(M_R)\cong R$, i.e.\ all $R$-linear maps $f:M\to M$
        are of the form $m\mapsto mr$ for some $r\in R$. Let $\alpha=\id_{R}\in\aAut(R)$.
        Then  $M\otimes_\alpha M$ is just $M\otimes_R M$. We now have:
        \begin{eqnarray*}
            b_\alpha(e_1,e_1) &=& e_1\otimes e_1 = e_2x \otimes e_1=e_2x\otimes  e_1=0\ ,\\
            b_\alpha(e_2,e_1) &=& e_2\otimes e_1 = e_2 \otimes  e_3y=  e_2y\otimes e_3=0\ , \\
            b_\alpha(e_3,e_1) &=& e_3\otimes e_1 = e_3 \otimes  e_2x =  e_3x\otimes e_2 =0\ .
        \end{eqnarray*}
        Therefore, $b_\alpha(M,e_1)=0$, hence $b_\alpha$ is not right injective.
        In particular, the correspondence \eqref{GEN:EQ:correspondence-adjusted-II} of Remark~\ref{GEN:RM:dfn-of-generic-forms} fails for $M$.
    \end{example}

    The next example demonstrates that $b_\alpha$ can be injective even when it is not regular.

    \begin{example}\label{GEN:EX:injective-forms-II}
        Let $F$ be a field, let $R=F[s,t]$ and let $M=\ideal{s,t}:=sR+tR$.
        It is routine to verify that $\End_R(M)\cong R$ (in the sense of the previous example). Let $\alpha=\id_R\in\aAut(\End_R(M))$. We claim that
        $b_\alpha$ is injective, but not regular. Indeed, we can
        identify $K_\alpha$ with $M\otimes_R M$ as above.
        Using this isomorphism,
        it is not hard (but tedious)
        to verify that the set
        \[\{s\otimes_\alpha s^i,~s\otimes_\alpha t,~t\otimes_\alpha s,~t\otimes_\alpha t^i,~s^j\otimes_\alpha t^k\where i,j,k\in \N,j+k>2\}\]
        forms an $F$-basis to $K_\alpha$.
        Let $\vphi\in M^{[1]}=\Hom_R(M,(K_\alpha)_0)$ be defined by $\vphi(x)=t\otimes_\alpha x$ (in fact $\vphi=\lAd{b_\alpha} t$).
        Then for all $x\in M$, we have $(\rAd{b_\alpha}x)s=s\otimes_\alpha x\neq t\otimes_\alpha s=\vphi(s)$, hence
        $\rAd{b_\alpha}x\neq \vphi$. It follows that $\rAd{b_\alpha}$ is not onto, hence $b_\alpha$ is not right regular.

        Now consider the classical form $b:M\times M\to R$ over $(R,\id_R)$ given
        by $b(x,y)=xy$. It is easy to check that $b$ is nondegenerate and $b(rx,y)=b(x,ry)$
        for all $r\in R$, hence there is a double $R$-module
        homomorphism $f:K_\alpha\to K$ such that $f\circ b_\alpha=b$. If there is $x\in M$
        such that $b_\alpha (M,x)=0$, then $b(M,x)=f(b_\alpha(M,x))=f(0)=0$, implying $x=0$.
        Thus, $b_\alpha$ is right injective. That $b_\alpha$ is left injective and not left regular follows by symmetry.
    \end{example}

    \begin{remark}
        Assume $b_\alpha$ is right injective but not right regular (e.g.\ as in Example~\ref{GEN:EX:injective-forms-II}).
        Then the correspondence
        in \eqref{GEN:EQ:correspondence-adjusted-II} fails. However, it is still possible to recover $\alpha$
        from $b_\alpha$ in this case, because for all $w\in W$, there exists a \emph{unique}  $w^\alpha\in W$ such that
        $b_\alpha(wx,y)=b_\alpha(x,w^\alpha y)$ for all $x,y\in M$. This suggests the following definition: Call a bilinear
        form $b:M\times M\to K$ \emph{right stable} if for all $w\in W:=\End_R(M)$, there exists a \emph{unique}
        $w^\alpha\in W$ such that $b(wx,y)=b(x,w^\alpha y)$ for all $x,y\in M$. The map $w\mapsto w^\alpha$
        is then a well-defined anti-endomorphism of $W$. We can now ask whether there is a correspondence
        between \emph{stable forms on $M$}, up to some similarity, and elements of $\aAut(W)$. It turns out that the answer is ``yes''
        in many cases, e.g.\ in the setting of  Example~\ref{GEN:EX:injective-forms-II}. Moreover, it is possible
        that $b_\alpha$ would be \emph{degenerate and stable}. The details of these results will be published elsewhere.
    \end{remark}

    The next two examples demonstrate what might happen when $M$ is a generator,
    but $\alpha\in\aEnd(\End(M))$ is not bijective. They imply that there need  not be a
    one-to-one correspondence between the anti-\emph{endo}morphisms of $\End(M)$ and the \emph{right regular}
    forms on $M$, despite the fact the the anti-\emph{auto}morphisms of $\End(M)$ correspond to \emph{regular}
    forms in this case (Theorem~\ref{GEN:TH:main-thm-II}). In addition,
    they imply that the injectivity of $\alpha$
    in Corollary \ref{GEN:CR:generic-forms-over-generators} is essential.

    \begin{example}\label{GEN:EX:example-IV}
        Let $N$ be any nonzero torsion $\Z$-module and let $M=\Z\oplus N\in\rMod{\Z}$. We consider the elements of $M$ as column vectors. Then
        \[W:=\End_{\Z}(M)=\SMatII{\End(\Z_\Z)}{\Hom(N,\Z)}{\Hom(\Z,N)}{\End(N_\Z)}=\SMatII{\Z}{0}{N}{\End(N_\Z)}\ .\]
        Note that $M$ is a generator and $e:=\smallSMatII{1}{0}{0}{0}\in W$ is a projection from
        $M$ onto $\Z_{\Z}$. Thus, we can identify $M$ with
        \[We=\SMatII{\Z}{0}{N}{0}\ .\]
        Define $\alpha\in \aEnd(W)$ by
        $\smallSMatII{a}{0}{b}{c}^\alpha=\smallSMatII{a}{0}{0}{\quo{a}}$ where $\quo{a}$ is the image
        of $a\in\Z$ in $\End(N_\Z)$.
        Then by Proposition \ref{GEN:PR:K-alpha-computation},
        $b_\alpha$ is similar to $b:M\times M\to e^\alpha We=1We=M$
        defined by $b(\smallSMatII{x}{0}{y}{0},\smallSMatII{z}{0}{w}{0}))=
        \smallSMatII{x}{0}{y}{0}^\alpha\smallSMatII{z}{0}{w}{0}=\smallSMatII{x}{0}{0}{\quo{x}}\smallSMatII{z}{0}{w}{0}=\smallSMatII{xz}{0}{xw}{0}$.
        It is now easy to see that $b$ is right injective but not left injective.
        In addition, $b$ is not right regular.
        To see this, let $0\neq f\in\End(N_{\Z})$ and note that the homomorphism
        $\smallSMatII{x}{0}{y}{0}\mapsto \smallSMatII{0}{0}{f(y)}{0}\in\Hom_{\Z}(M,(K_\alpha)_0)$ does
        not lie in  $\im(\rAd{b})$.
    \end{example}

    \begin{example}\label{GEN:EX:example-V}
        View $N:=\Z[\frac{1}{p}]/\Z$ as a $\Z_p$-module as in Example \ref{GEN:EX:counter-example-I}.
        Then $\End(N_{\Z_p})=\Z_p$. Define $M=\Z_p\oplus N\in\rMod{\Z_p}$ and consider the elements
        of $M$ as column vectors.
        Then
        \[W:=\End_{{\Z}_p}(M)=\SMatII{\End(\Z_p)}{\Hom(N,\Z_p)}{\Hom(\Z_p,N)}{\End(N)}=\SMatII{\Z_p}{0}{N}{\Z_p}\ .\]
        Let $e=\smallSMatII{1}{0}{0}{0}\in W$. As in the previous example, we can identify $M$ with $We$.
        Define $\alpha\in\aEnd(W)$ by $\smallSMatII{x}{0}{y}{z}=\smallSMatII{z}{0}{0}{z}$.
        Then by Proposition \ref{GEN:PR:K-alpha-computation},
        $K_\alpha$ is isomorphic to $e^\alpha We=0We=0$,
        so $b_\alpha$ is the zero form!\rem{(In particular, the assumption that $\alpha$ is injective in Corollary
        \ref{GEN:CR:generic-forms-over-generators} is essential).}
    \end{example}

    Our last example demonstrates that a regular bilinear form $b$ need not be
    similar to its generization, $b_{\alpha(b)}$. The construction
    also shows several other interesting phenomena, which are pointed out at the end.

    \begin{example}\label{GEN:EX:b-is-not-similar-to-b-alpha}
        Let $1<n\in\N$ and let $F$ be a field. Denote by $R$ the ring of upper-triangular matrices
        over $F$. For $0\leq i\leq n$, let $M_i$ denote the right $R$-module consisting
        of row vectors of the form
        \[(\,0,\dots,0,\underbrace{*,\dots,*}_i\,)\in F^n\]
        with $R$ acting by matrix multiplication on the right.
        It is not hard to verify that $\dim_F\Hom_{R}(M_n,M_n/M_{i})=1$
        for all $0\leq i<n$. In addition, $\End_{R}(M_n/M_i)=F$ for
        all $0\leq i<n$. The latter implies the following fact, to be used later,
        which follows from the Krull-Schmidt Theorem (see
        \cite[Th.\ 2.9.17]{Ro88}): Assume that $a_0+\dots +a_n=b_0+\dots +b_n$ and
        $\bigoplus_{i=0}^n (M_n/M_i)^{a_i}\cong\bigoplus_{i=0}^n (M_n/M_i)^{b_i}$. Then
        $a_i=b_i$ for all $0\leq i\leq n$. (The assumption $a_0+\dots+ a_n=b_0+\dots +b_n$
        is needed because $M_n/M_n$ is the zero module.)

        Let $T$ be the transpose involution on $K:=\nMat{F}{n}$. We make $K$
        into  a double $R$-module by defining
        \[A\mul{0} B=\trans{B}A\qquad\textrm{and}\qquad A\mul{1} B=AB\]
        for all $A\in K$ and $B\in R$. Then $b:M_n\times M_n\to K$
        defined by $b(x,y)=\trans{x}y$ is a bilinear form. For $0\leq u,v\leq n$,
        let $K_{u,v}$ denote the matrices $A=(A_{ij})\in K$ for which
        $A_{ij}=0$ if $i\leq u$ or $j\leq v$. For example,
        when $n=3$, $K_{1,2}$ and $K_{2,0}$ consist of matrices of
        the forms
        \[\SMatIII{0}{0}{0}{0}{0}{*}{0}{0}{*}\quad,\quad \SMatIII{0}{0}{0}{0}{0}{0}{*}{*}{*}\]
        respectively.
        Then $K_{u,v}$ is a sub-double-$R$-module
        of $K$, hence $K/K_{u,v}$ is a double $R$-module, and $b_{u,v}:M_n\times  M_n\to K/K_{u,v}$
        defined by $b(x,y)=\trans{x}y+K_{u,v}$ is a bilinear form.

        Recall that we use $K_i$ to denote the $R$-module obtained from $K$ by letting $R$ act via $\mul{i}$.
        Then $(K/K_{u,v})_0\cong M_n^v\oplus (M_n/M_{n-u})^{n-v}$ (the summands
        are the columns
        of $K/K_{u,v}=\nMat{F}{n}/K_{u,v}$) and $(K/K_{u,v})_1\cong M_n^u\oplus (M_n/M_{n-v})^{n-u}$ (the summands
        are the rows
        of $K/K_{u,v}=\nMat{F}{n}/K_{u,v}$). Therefore, by the fact recorded above, for $0\leq u,u',v,v'<n$ and $i\in\{0,1\}$,
        we have
        \begin{eqnarray*}
        (K/K_{u,v})_i\cong (K/K_{u',v'})_i\phantom{{}_{-1}}&\iff& (u,v)=(u',v') \qquad \textrm{and} \\
        (K/K_{u,v})_i\cong (K/K_{u',v'})_{1-i}&\iff&
        (u,v)=(v',u')\ .
        \end{eqnarray*}

        We now claim that $b_{u,v}$ is right regular when $u>0$.
        Indeed,
        it is easy to check that $b_{u,v}$ is right injective in this case. In addition,
        we have
        \begin{eqnarray*}
        \dim_F M^{[1]}&=&
        \dim_F\Hom_{R}(M_n,(K/K_{u,v})_0) \\
        &=&v\dim_F\Hom_{R}(M_n,M_n)+(n-v)\dim_F\Hom_{R}(M_n,M_n/M_{n-u})=n\ .
        \end{eqnarray*}
        Therefore, dimension considerations imply $\rAd{b_{u,v}}$ is bijective, i.e.\ $b_{u,v}$ is right regular.
        Similarly, $b_{u,v}$ is left regular when $v>0$.

        Now, observe that $\End_{R}(M_n)\cong F$ and $b_{u,v}(ax,y)=b_{u,v}(x,ay)$ for all $x,y\in M_n$ and $a\in F$. Therefore, provided
        $b_{u,v}$ is right regular,
        its corresponding anti-endomorphism is $\id_F$.
        It follows that the forms
        $\{b_{u,v}\where 0<u\}$ are right regular and have the same corresponding
        anti-automorphism (which is in fact an involution), and hence the
        same generization, which is regular since $M_n$ is f.g.\ projective (Corollary~\ref{GEN:CR:main-thm-I}).
        In addition, the double $R$-modules $\{K/K_{u,v}\where u,v<n\}$
        are pairwise non-isomorphic, hence the forms
        $\{b_{u,v}\where u,v<n\}$ are pairwise non-similar.
        This means that
        the forms $\{b_{u,v}\where 0<u,v<n\}$ are regular, pairwise non-similar, but
        nevertheless have the same generization, which is regular. In particular, necessarily all
        but possibly one of them is not similar to its generization.
        Let us show that this is in fact true for all of these forms.

        We claim that common generization of $\{b_{u,v}\where 0<u\}$ is similar to $b$.
        Indeed, let $\alpha=\id_F\in\aEnd(F)$. Then $\dim_F K_\alpha=\dim_F M_n\otimes_\alpha M_n=\dim_F
        M_n\otimes_F M_n=n^2$. The universality of $b_\alpha$ implies that there is a double
        $R$-module homomorphism $f:K_\alpha\to K$, that must be onto since $\im(b)=K$.
        (Recall that $\im(b)$ was defined to be the additive group spanned by $\{b(x,y)\where x,y\in M_n\}$.)
        As $f$ is clearly $F$-linear and $\dim_F K=n^2$, dimension considerations imply that $f$ is an
        isomorphism. Thus, $b$ is the generization of all the forms $\{b_{u,v}\where 0<u\}$, and since
        $\dim_FK>\dim_F K/K_{u,v}$ whenever $u,v<n$, we see that $b_{u,v}\nsim b$ for all $0<u,v<n$, as required.

        Note that since $M_n$ is f.g.\ projective, there is a one-to-one correspondence between
        $\aEnd(\End_R(M_n))$ and the \emph{right generic} forms on $M_n$, up to similarity (Remark~\ref{GEN:RM:dfn-of-generic-forms}
        and Corollary~\ref{GEN:CR:main-thm-I}). However, we have just shown that there is no
        correspondence between $\aEnd(\End_R(M_n))$ the \emph{right regular} forms on $M_n$ (since
        the maps $\alpha\mapsto b_\alpha$ and $b\mapsto \alpha(b)$ of section~\ref{section:GEN:preface} are
        not  inverse to each other).

        Finally, let $0<u,v<n$ be distinct. Then $b_{u,v}$ is regular, $\alpha(b_{u,v})$ is an involution, but
        $\im(b)=K/K_{u,v}$ does not have an anti-automorphism since $(K/K_{u,v})_0\ncong (K/K_{u,v})_1$.
        Furthermore, $b_{u,0}$ is right regular, but left
        degenerate. This shows that
        Proposition~\ref{GEN:PR:generic-forms-basic-props} fails for non-generic forms.
    \end{example}

    We could neither find nor contradict the existence of the following:
    \begin{itemize}
        \item An anti-\emph{auto}morphism $\alpha$ such that $b_\alpha$
        is right regular but not left regular. (In this case $\alpha^2$ cannot be inner
        by Corollary \ref{GEN:CR:right-regular-iff-left-regular}.)
        \item A generator $M$ and an \emph{injective} $\alpha\in\aEnd(\End_R(M))$ such that
        $b_\alpha$ is not \emph{right} injective.
    \end{itemize}

\section*{Acknowledgments}

    I deeply thank to my supervisor, Uzi Vishne, for guiding me through the research\rem{,
    to David Saltman for suggesting the problem discussed at section
    \ref{section:FORM:applications}}, to Lance Small for a very
    beneficial conversation, and to Jean-Pierre Tignol for his comments on earlier versions.

\bibliographystyle{plain}
\bibliography{MyBib}

\end{document}

%% file: DefaultAMS.tex
\usepackage{amsfonts}
\usepackage{amssymb}
\usepackage{amsmath}
\usepackage{hyperref}
\usepackage{mathrsfs}
\usepackage{centernot}
\usepackage{mathdots}
\usepackage{stmaryrd}
\usepackage[all]{xy}

\newtheorem{thm}{Theorem}[section]
\newtheorem{lem}[thm]{Lemma}
\newtheorem{prp}[thm]{Proposition}
\newtheorem{cor}[thm]{Corollary}
\newtheorem{dfn}[thm]{Definition}

\newtheorem{que}{Question}

\newtheorem{baseexample}[thm]{Example} %never to be used !!! - use example environment
\newtheorem{baseremark}[thm]{Remark} %never to be used !!! - use remark environment

\newenvironment{example}
{\begin{baseexample}\rm}{\end{baseexample}}
\newenvironment{remark}
{\begin{baseremark}\rm}{\end{baseremark}}

\newcommand{\rem}[1]{}

\newcommand{\N}{\mathbb{N}}
\newcommand{\Q}{\mathbb{Q}}

\newcommand{\Z}{\mathbb{Z}}

\newcommand{\frakSmall}{
\newcommand{\fraka}{{\mathfrak{a}}}
\newcommand{\frakb}{{\mathfrak{b}}}
\newcommand{\frakc}{{\mathfrak{c}}}
\newcommand{\frakd}{{\mathfrak{d}}}
\newcommand{\frake}{{\mathfrak{e}}}
\newcommand{\frakf}{{\mathfrak{f}}}
\newcommand{\frakg}{{\mathfrak{g}}}
\newcommand{\frakh}{{\mathfrak{h}}}
\newcommand{\fraki}{{\mathfrak{i}}}
\newcommand{\frakj}{{\mathfrak{j}}}
\newcommand{\frakk}{{\mathfrak{k}}}
\newcommand{\frakl}{{\mathfrak{l}}}
\newcommand{\frakm}{{\mathfrak{m}}}
\newcommand{\frakn}{{\mathfrak{n}}}
\newcommand{\frako}{{\mathfrak{o}}}
\newcommand{\frakp}{{\mathfrak{p}}}
\newcommand{\frakq}{{\mathfrak{q}}}
\newcommand{\frakr}{{\mathfrak{r}}}
\newcommand{\fraks}{{\mathfrak{s}}}
\newcommand{\frakt}{{\mathfrak{t}}}
\newcommand{\fraku}{{\mathfrak{u}}}
\newcommand{\frakv}{{\mathfrak{v}}}
\newcommand{\frakw}{{\mathfrak{w}}}
\newcommand{\frakx}{{\mathfrak{x}}}
\newcommand{\fraky}{{\mathfrak{y}}}
\newcommand{\frakz}{{\mathfrak{z}}}
}

\newcommand{\calCapital}{
\newcommand{\calA}{{\mathcal{A}}}
\newcommand{\calB}{{\mathcal{B}}}
\newcommand{\calC}{{\mathcal{C}}}
\newcommand{\calD}{{\mathcal{D}}}
\newcommand{\calE}{{\mathcal{E}}}
\newcommand{\calF}{{\mathcal{F}}}
\newcommand{\calG}{{\mathcal{G}}}
\newcommand{\calH}{{\mathcal{H}}}
\newcommand{\calI}{{\mathcal{I}}}
\newcommand{\calJ}{{\mathcal{J}}}
\newcommand{\calK}{{\mathcal{K}}}
\newcommand{\calL}{{\mathcal{L}}}
\newcommand{\calM}{{\mathcal{M}}}
\newcommand{\calN}{{\mathcal{N}}}
\newcommand{\calO}{{\mathcal{O}}}
\newcommand{\calP}{{\mathcal{P}}}
\newcommand{\calQ}{{\mathcal{Q}}}
\newcommand{\calR}{{\mathcal{R}}}
\newcommand{\calS}{{\mathcal{S}}}
\newcommand{\calT}{{\mathcal{T}}}
\newcommand{\calU}{{\mathcal{U}}}
\newcommand{\calV}{{\mathcal{V}}}
\newcommand{\calW}{{\mathcal{W}}}
\newcommand{\calX}{{\mathcal{X}}}
\newcommand{\calY}{{\mathcal{Y}}}
\newcommand{\calZ}{{\mathcal{Z}}}
}

\newcommand{\bbCapital}{
\newcommand{\bbA}{{\mathbb{A}}}
\newcommand{\bbB}{{\mathbb{B}}}
\newcommand{\bbC}{{\mathbb{C}}}
\newcommand{\bbD}{{\mathbb{D}}}
\newcommand{\bbE}{{\mathbb{E}}}
\newcommand{\bbF}{{\mathbb{F}}}
\newcommand{\bbG}{{\mathbb{G}}}
\newcommand{\bbH}{{\mathbb{H}}}
\newcommand{\bbI}{{\mathbb{I}}}
\newcommand{\bbJ}{{\mathbb{J}}}
\newcommand{\bbK}{{\mathbb{K}}}
\newcommand{\bbL}{{\mathbb{L}}}
\newcommand{\bbM}{{\mathbb{M}}}
\newcommand{\bbN}{{\mathbb{N}}}
\newcommand{\bbO}{{\mathbb{O}}}
\newcommand{\bbP}{{\mathbb{P}}}
\newcommand{\bbQ}{{\mathbb{Q}}}
\newcommand{\bbR}{{\mathbb{R}}}
\newcommand{\bbS}{{\mathbb{S}}}
\newcommand{\bbT}{{\mathbb{T}}}
\newcommand{\bbU}{{\mathbb{U}}}
\newcommand{\bbV}{{\mathbb{V}}}
\newcommand{\bbW}{{\mathbb{W}}}
\newcommand{\bbX}{{\mathbb{X}}}
\newcommand{\bbY}{{\mathbb{Y}}}
\newcommand{\bbZ}{{\mathbb{Z}}}
}

\newcommand{\what}[1]{\widehat{#1}}

\newcommand{\veps}{\varepsilon}
\newcommand{\vphi}{\varphi}
\newcommand{\lmb}{\lambda}
\newcommand{\Lmb}{\Lambda}

\newcommand{\idealof}{\unlhd} % \vartriangleleft

\newcommand{\derives}{\Longrightarrow}
\newcommand{\nderives}{\centernot\Longrightarrow}

\newcommand{\suchthat}{\,:\,}
\newcommand{\where}{\,|\,}

\newcommand{\quo}[1]{\overline{#1}}

\newcommand{\SMatII}[4]{\left[\begin{array}{cc} {#1} & {#2} \\ {#3} &
{#4} \end{array}\right]}
\newcommand{\smallSMatII}[4]{\left[\begin{smallmatrix} {#1} & {#2} \\ {#3} &
{#4} \end{smallmatrix}\right]}
\newcommand{\SMatIII}[9]{\left[\begin{array}{ccc} {#1} & {#2} & {#3} \\ {#4} &
{#5} & {#6} \\ {#7} & {#8} & {#9} \end{array}\right]}

\newcommand{\DotsArr}[4]{
\begin{array}{ccc} {#1} & \ldots & {#2} \\
\vdots & \ddots & \vdots \\
{#3} & \ldots & {#4} \end{array}}

\newcommand{\Trings}[1]{\left< #1 \right>}

\DeclareMathOperator{\ann}{ann}
\newcommand{\annl}{\ann^\ell}
\newcommand{\annr}{\ann^r}
 \DeclareMathOperator{\Aut}{Aut}
\DeclareMathOperator{\Bil}{Bil} %
\DeclareMathOperator{\Br}{Br} %
\DeclareMathOperator{\Cent}{Cent}
\DeclareMathOperator{\Char}{char}

\DeclareMathOperator{\Cor}{Cor}
 \DeclareMathOperator{\End}{End}
\DeclareMathOperator{\Gal}{Gal} 
\DeclareMathOperator{\Hom}{Hom} %
\DeclareMathOperator{\id}{id} %
\DeclareMathOperator{\im}{im} %
\DeclareMathOperator{\Inn}{Inn} %
\DeclareMathOperator{\Jac}{Jac} %

\newcommand{\op}{\mathrm{op}} %
\DeclareMathOperator{\Out}{Out} %
\DeclareMathOperator{\Spec}{Spec} %
\DeclareMathOperator{\udim}{u.dim}

\newcommand{\nMat}[2]{\mathrm{M}_{#2}(#1)}

\newcommand{\trans}[1]{#1^{{T}}\!}

\newcommand{\dirlim}{\underrightarrow{\lim}\,}

\newcommand{\units}[1]{{#1^\times}}
\newcommand{\ideal}[1]{\left<#1\right>}

\newcommand{\rMod}[1]{{\mathrm{Mod}\textrm{-}{#1}}}

\newcommand{\lMod}[1]{{#1}\textrm{-}{\mathrm{Mod}}}
\newcommand{\biMod}[2]{{\mathrm{Mod}\textrm{-}({#1},{#2})}}

\newcommand{\rproj}[1]{{\mathrm{proj}}\textrm{-}{#1}}

%% file: NonClassicalForms_v10.bbl
\def\Dbar{\leavevmode\lower.6ex\hbox to 0pt{\hskip-.23ex \accent"16\hss}D}
  \def\Dbar{\leavevmode\lower.6ex\hbox to 0pt{\hskip-.23ex \accent"16\hss}D}
  \def\Dbar{\leavevmode\lower.6ex\hbox to 0pt{\hskip-.23ex \accent"16\hss}D}
  \def\Dbar{\leavevmode\lower.6ex\hbox to 0pt{\hskip-.23ex \accent"16\hss}D}
  \def\Dbar{\leavevmode\lower.6ex\hbox to 0pt{\hskip-.23ex \accent"16\hss}D}
\begin{thebibliography}{10}

\bibitem{Fi13B}
Uriya~A. First.
\newblock Azumaya algebras with involution.
\newblock Preprint (currently avaiable at http://arxiv.org/abs/1305.5139),
  2013.

\bibitem{Ka82ModulesAndRings}
F.~Kasch.
\newblock {\em Modules and rings}, volume~17 of {\em London Mathematical
  Society Monographs}.
\newblock Academic Press Inc. [Harcourt Brace Jovanovich Publishers], London,
  1982.
\newblock Translated from the German and with a preface by D. A. R. Wallace.

\bibitem{Kn91}
Max-Albert Knus.
\newblock {\em Quadratic and {H}ermitian forms over rings}, volume 294 of {\em
  Grundlehren der Mathematischen Wissenschaften [Fundamental Principles of
  Mathematical Sciences]}.
\newblock Springer-Verlag, Berlin, 1991.
\newblock With a foreword by I. Bertuccioni.

\bibitem{InvBook}
Max-Albert Knus, Alexander Merkurjev, Markus Rost, and Jean-Pierre Tignol.
\newblock {\em The book of involutions}, volume~44 of {\em American
  Mathematical Society Colloquium Publications}.
\newblock American Mathematical Society, Providence, RI, 1998.
\newblock With a preface in French by J. Tits.

\bibitem{La99}
T.~Y. Lam.
\newblock {\em Lectures on modules and rings}, volume 189 of {\em Graduate
  Texts in Mathematics}.
\newblock Springer-Verlag, New York, 1999.

\bibitem{Laz64}
Daniel Lazard.
\newblock Sur les modules plats.
\newblock {\em C. R. Acad. Sci. Paris}, 258:6313--6316, 1964.

\bibitem{Ma58}
Eben Matlis.
\newblock Injective modules over {N}oetherian rings.
\newblock {\em Pacific J. Math.}, 8:511--528, 1958.

\bibitem{Osborn70}
J.~Marshall Osborn.
\newblock Jordan and associative rings with nilpotent and invertible elements.
\newblock {\em J. Algebra}, 15:301--308, 1970.

\bibitem{QuSchSch79}
H.-G. Quebbemann, W.~Scharlau, and M.~Schulte.
\newblock Quadratic and {H}ermitian forms in additive and abelian categories.
\newblock {\em J. Algebra}, 59(2):264--289, 1979.

\bibitem{MaximalOrders}
I.~Reiner.
\newblock {\em Maximal orders}, volume~28 of {\em London Mathematical Society
  Monographs. New Series}.
\newblock The Clarendon Press Oxford University Press, Oxford, 2003.
\newblock Corrected reprint of the 1975 original, With a foreword by M. J.
  Taylor.

\bibitem{Reiter75}
H.~Reiter.
\newblock Witt's theorem for noncommutative semilocal rings.
\newblock {\em J. Algebra}, 35:483--499, 1975.

\bibitem{RiSh76}
C.~Riehm and M.~A. Shrader-Frechette.
\newblock The equivalence of sesquilinear forms.
\newblock {\em J. Algebra}, 42(2):495--530, 1976.

\bibitem{Ri74}
Carl Riehm.
\newblock The equivalence of bilinear forms.
\newblock {\em J. Algebra}, 31:45--66, 1974.

\bibitem{Ro88}
Louis~H. Rowen.
\newblock {\em Ring theory. {V}ol. {I}}, volume 127 of {\em Pure and Applied
  Mathematics}.
\newblock Academic Press Inc., Boston, MA, 1988.

\bibitem{Sa99}
David~J. Saltman.
\newblock {\em Lectures on division algebras}, volume~94 of {\em CBMS Regional
  Conference Series in Mathematics}.
\newblock Published by American Mathematical Society, Providence, RI, 1999.

\bibitem{ScharR81}
Rudolf Scharlau.
\newblock Zur {K}lassifikation von {B}ilinearformen und von {I}sometrien \"uber
  {K}\"orpern.
\newblock {\em Math. Z.}, 178(3):359--373, 1981.

\bibitem{SchQuadraticAndHermitianForms}
Winfried Scharlau.
\newblock {\em Quadratic and {H}ermitian forms}, volume 270 of {\em Grundlehren
  der Mathematischen Wissenschaften [Fundamental Principles of Mathematical
  Sciences]}.
\newblock Springer-Verlag, Berlin, 1985.

\end{thebibliography}
